\documentclass{amsart}
\usepackage[foot]{amsaddr}

\usepackage[english]{babel}
\usepackage[square,numbers,sort]{natbib}
\usepackage[letterpaper,top=2cm,bottom=2cm,left=3cm,right=3cm,marginparwidth=1.75cm]{geometry}

\usepackage{amsmath}
\usepackage{amssymb}
\usepackage{graphicx}
\usepackage{xcolor}
\usepackage[colorlinks=true, allcolors=blue]{hyperref}
\usepackage{mathrsfs}
\usepackage{xfrac}
\usepackage{xfrac}  
\usepackage{breqn}
\usepackage{dsfont}
\usepackage{tikz}
\usetikzlibrary{calc}
\usepackage{esint}
\usepackage{algorithm}
\usepackage[noend]{algpseudocode}

\makeatletter
\def\namedlabel#1#2{\begingroup
	#2%
	\def\@currentlabel{#2}%
	\phantomsection\label{#1}\endgroup
}
\def\BState{\State\hskip-\ALG@thistlm}
\makeatother

\theoremstyle{plain}
\newtheorem{theorem}{Theorem}[section]
\newtheorem{lemma}{Lemma}[section]
\newtheorem{corollary}{Corollary}[section]

\newtheorem{proposition}{Proposition}[section]
\newtheorem{remark}{Remark}[section]

\theoremstyle{definition}

\newtheorem{example}{Example}[section]

\newcommand{\Pcal}{\mathcal{P}}

\newcommand{\Fcal}{\mathcal{F}}

\newcommand{\Mcal}{\mathcal{M}}
\newcommand{\Ncal}{\mathcal{N}}
\newcommand{\Rbb}{\mathbb{R}}
\newcommand{\Nbb}{\mathbb{N}}  % \Nbb = \{1,2,3,\ldots\}
\newcommand{\Pbb}{\mathbb{P}}
\newcommand{\Ebb}{\mathbb{E}}
\newcommand {\Id}{I}

\newcommand{\Frechet}{Fr\'echet }
\newcommand{\diff}{\,\textnormal{d}}
\newcommand{\CovSpace}{\mathbb{K}}
\newcommand{\CovSpacePos}{\mathbb{K}_{+}}
\newcommand{\SymSpace}{\mathbb{S}}
\newcommand{\id}{\textnormal{id}}
\newcommand{\Tan}{\operatorname{Tan}}
\newcommand{\Hilbert}{\mathcal{H}}

\newcommand{\BW}{\Pi}

\newcommand{\effdom}{\mathcal{D}_P}

\DeclareMathOperator*{\argmin}{arg\,min} 
\DeclareMathOperator*{\trace}{tr} 
\DeclareMathOperator{\proj}{proj}

\DeclareMathOperator{\wass}{\mathcal{W}_{\textnormal{ac}}}
\newcommand{\weak}{\tau_{\textnormal{w}}}

\newcommand{\bary}{M^{\ast}}
\newcommand{\empbary}{M_n^{\ast}}

\makeatletter
\def\namedlabel#1#2{\begingroup
  #2%
  \def\@currentlabel{#2}%
  \phantomsection\label{#1}\endgroup
}
\makeatother

\title[Large Deviations Principle for Bures-Wasserstein Barycenters]{Large Deviations Principle for~\\Bures-Wasserstein Barycenters}

\author{Adam Quinn Jaffe}
\address[AQJ]{Department of Statistics, Columbia University, New York, USA}

\author{Leonardo V. Santoro}
\address[LVS]{Institut de Math\'ematiques, École Polytechnique Fédérale de Lausanne, Lausanne, Switzerland}

\email{a.q.jaffe@columbia.edu}
\email{leonardo.santoro@epfl.ch}

\thanks{This material is based upon work for which LVS was supported by a Swiss National Science Foundation (SNF) grant.}

	

\begin{document}

\begin{abstract}
    We prove the large deviations principle for empirical Bures-Wasserstein barycenters of independent, identically-distributed samples of covariance matrices and covariance operators.
    As an application, we explore some consequences of our results for the phenomenon of dimension-free concentration of measure for Bures-Wasserstein barycenters.
    Our theory reveals a novel notion of exponential tilting in the Bures-Wasserstein space, which, in analogy with Cram\'er's theorem in the Euclidean case, solves the relative entropy projection problem under a constraint on the barycenter.    
    Notably, this method of proof is easy to adapt to other geometric settings of interest; with the same method, we obtain large deviations principles for empirical barycenters in Riemannian manifolds and the univariate Wasserstein space, and we obtain large deviations upper bounds for empirical barycenters in the general multivariate Wasserstein space.
    In fact, our results are the first known large deviations principles for Fr\'echet means in any non-linear metric space.
\end{abstract}

\maketitle

\setcounter{tocdepth}{2}
\tableofcontents

\section{Introduction}

\subsection{Setting}
In varied application areas from imaging to biology, many modern statistical problems involve estimation and inference for data which are \textit{covariance operators} \cite{alexander2005multiple,schwartzman2008false,friston2011functional,dai2019analyzing,panaretos2010second, Pigoli,bhatia2019bures,delBarrioClustering,Gonzalez,Dryden};
these are defined as non-negative, self-adjoint, trace-class operators, but in suitable finite-dimensional settings they are simply equivalent to positive semi-definite matrices.
Because of the non-Euclidean structure of the space $\CovSpace$ of covariance operators, one is typically interested in developing statistical methodologies that are based on some sort of intrinsic geometry of $\CovSpace$; however, there are many interesting geometries that can be placed on this space, each with some advantages and disadvantages, from a statistical point of view \cite{Dryden, Pigoli}.

For many problems, a natural geometry comes from the \textit{Bures-Wasserstein metric} $\BW$ on $\CovSpace$,
\begin{equation}
    \label{eqn:BWdef}
\BW(\Sigma,\Sigma') := \sqrt{ \trace(\Sigma)+\trace(\Sigma')-2\trace(({\Sigma}^{\sfrac{1}{2}}\Sigma'{\Sigma}^{\sfrac{1}{2}})^{\sfrac{1}{2}})},
\end{equation}
which makes $\CovSpace$ into a complete geodesic metric space of non-negative (Alexandrov) curvature. (See \cite{bhatia2019bures, Masarotto,Kroshnin, Dryden,Pigoli, OlkinPukelsheim, CuestaAlbertos}.)
Interestingly, the Bures-Wasserstein metric has been independently studied in several contexts, through different but equivalent formulations:
In quantum information theory, one restricts attention to the subset of $\Sigma\in\CovSpace$ with $\trace(\Sigma) = 1$, in which case $\BW$ determines a natural distance on \textit{(mixed) states} of a quantum system \cite{bhatia2019bures}.
In statistical shape analysis, one considers the distance between two covariance operators to be the minimal distance between any two possible respective (operator-theoretic) square roots, leading to the so-called \textit{Procrustes distance} \cite{Pigoli}.
In optimal transport (which is perhaps the most important setting for our work), one often restricts the \textit{Wasserstein space} of probability measures to the subspace of Gaussian measures, and in this case the Bures-Wasserstein distance arises from the inherited geometry \cite{Masarotto}; this also endows the Bures-Wasserstein space, formally speaking, with a Riemannian structure.

For a fixed but unknown probability measure $P$ on $\CovSpace$, we define a \textit{population (Bures-Wasserstein) barycenter} to be any solution $\bary$ to the optimization problem
\begin{equation}\label{eqn:pop-bary}
    \begin{cases}
        \textnormal{minimize} & \int_{\CovSpace}\BW^2(M,\Sigma)\diff P(\Sigma) \\
        \textnormal{over} & M\in \CovSpace.
    \end{cases}
\end{equation}
A fundamental statistical problem in the Bures-Wasserstein space is estimating the population barycenter on the basis of independent, identically distributed (IID) samples $\Sigma_1,\Sigma_2,\ldots$ from $P$.
A natural estimator is given by an \textit{empirical (Bures-Wasserstein) barycenter} which is defined to be any solution $\empbary$ to the optimization problem
\begin{equation}\label{eqn:emp-bary}
    \begin{cases}
        \textnormal{minimize} & \frac{1}{n}\sum_{i=1}^{n}\BW^2(M,\Sigma_i) \\
        \textnormal{over} & M\in \CovSpace.
    \end{cases}
\end{equation}
Here, the barycenter is a canonical notion of central tendency which extends the notion of Euclidean mean to the non-linear metric space $(\CovSpace, \BW)$; in different parts of the literature, the barycenter is also called the \textit{\Frechet mean}, the \textit{intrinsic mean}, or the \textit{center of mass}.
The definitions above can be generalized mutatis mutandis to abstract metric spaces \cite{frechet1948elements}, and much is known about the convergence of the empirical barycenter to the population barycenter \cite{EvansJaffeSLLN,Bhattacharya2003,Bhattacharya2005,WassersteinBarycenter,le2017existence}.

There are many probabilistic limit theorems for Bures-Wasserstein barycenters which are counterparts to the usual results in Euclidean spaces.
Indeed, one has the strong law of large numbers (SLLN) \cite{WassersteinBarycenter, PanaretosSantoroBW}, the central limit theorem (CLT) \cite{PanaretosSantoroBW, Kroshnin} and various finite-sample concentration inequalities \cite{Kroshnin, AGP, WassersteinConvergence}.
The goal of this paper is to develop a different concept of probabilistic convergence, namely the rate of decay of rare events related to the convergence of $\empbary$ to $\bary$; this is of course the domain of large deviations theory.

\subsection{Results and Outline}
Our main result (Theorem~\ref{thm:main}) shows that, for any Borel set $E\subseteq \CovSpace$, we have
\begin{equation*}
    \begin{split}
        -\inf\{I_P(M): M\in E^{\circ}\} &\le \liminf_{n\to\infty}\frac{1}{n}\log \Pbb(\empbary\in E) \\
        &\le \limsup_{n\to\infty}\frac{1}{n}\log \Pbb(\empbary\in E) \le -\inf\{I_P(M): M\in \bar E\},
    \end{split}
\end{equation*}
where $E^{\circ}$ and $\bar E$ denote the interior and closure of $E$ with respect to $\BW$, respectively, and the function $I_P:\CovSpace\to [0,\infty]$ is defined via
\begin{equation*}
         I_P(M) = \sup_{A\in\SymSpace_M}\left(\trace(AM) - \log\int_{\CovSpace}\exp \trace\left(AM t_{M}^{\Sigma}\right)\diff P(\Sigma)\right),
     \end{equation*}
     for $M\in\CovSpace$; here, $\SymSpace_M$ is the set of all self-adjoint operators $A$ on $\Hilbert$ satisfying $\|M^{\sfrac{1}{2}}A\|_2 < \infty$, and $t_M^{\Sigma}$ is the optimal transport map from $M$ to $\Sigma$. (See Section~\ref{sec:prelim} for precise definitions and geometric interpretations.)
     In classical terms, we prove that the empirical Bures-Wasserstein barycenters satisfy a large deviations principle with good rate function $I_P$.
     This can be seen as an analog of the many large deviations results for linear means, including Cram\'er's theorem \cite[Sections~2.2 and 2.3]{dembo2009large}, Sanov's theorem \cite[Theorem~6.2.10]{dembo2009large}, Schilder's theorem \cite[Theorem~5.2.3]{dembo2009large}, and more \cite{Bahadur79}.
    In fact, our result is the first ever large deviations principle for barycenters in any non-linear metric space.

    For the uninitiated reader, we note that the large deviations principle can be understood in more concrete terms for suitable sets $E\subseteq \CovSpace$.
    That is, if $E$ is equal to the closure of its interior, then we have
    \begin{equation*}
        \Pbb(\empbary\in E) \approx \exp(-I_P(E)n)
    \end{equation*}
    for large $n\in\Nbb$, ignoring sub-exponential factors, where
    \begin{equation}\label{eqn:numerical}
        I_P(E):=\inf\{I_P(M):M\in E\}.
    \end{equation}
    In other words, the minimization of the rate function $I_P$ exactly determines the exponential rate of decay of rare events related to the convergence $M_n^{\ast} \to M^{\ast}$, as $n\to\infty$.
    Because of this, one may be interested in numerically solving \eqref{eqn:numerical}, and we indeed address numerical considerations in Subsection~\ref{subsec:numerical}.

    The new insight (Proposition~\ref{prop:tilting}) that allows us to tackle the large deviations theory is noticing that $I_P(M)$ can be equivalently formulated in terms of the following relative entropy projection problem with barycenter constraint,
    \begin{equation*}
        \begin{cases}
			\textnormal{minimize} &H(Q\, | \, P) \\
			\textnormal{over } & Q\in \Pcal_2(\CovSpace) \\
			\textnormal{with} & Q \textnormal{ has barycenter } M,
		\end{cases}
    \end{equation*}
    where $H(Q\,|\,P)$ denotes the relative entropy (also called the Kullback-Leibler divergence) of $Q$ from $P$.
    In fact, we show that this problem is solved by a suitable transformation of $P$ which mirrors the usual Euclidean notion of exponential tilting.
    This is the key ingredient of the proof of our main result (Theorem~\ref{thm:main}) and it furthermore allows us to establish some fundamental properties of the rate function $I_P$ (Proposition~\ref{prop:rate-function-properties}).

    Our general results have interesting consequences for the study of rare events related to the convergence of empirical Bures-Wasserstein barycenters; we give two examples of this which are focused on the phenomenon of concentration of measure for barycenters.
    Indeed, we show (Corollary~\ref{cor:Hoeffding}) that the empirical barycenters asymptotically satisfy a dimension-free Hoeffding-type inequality and we characterize (Corollary~\ref{cor:degenerate}) the exact mechanism by which large distances $\BW(M_n^{\ast},M^{\ast})$ may occur.
    
    Lastly, we mention that our novel method of proof can be applied to prove large deviations results for barycenters in other geometric settings, and we give a few examples of this in Section~\ref{sec:ext}.
    Indeed, it yields a bona fide large deviations principle for barycenters in Riemannian manifolds and for barycenters in the univariate Wasserstein space, as well as large deviations upper bounds for the multivariate Wasserstein space.

\subsection{Related Literature}

Our results contribute to a rich literature on probabilistic limit theorems for Bures-Wasserstein barycenters \cite{WassersteinBarycenter,WassersteinConvergence, Kroshnin,Masarotto,BW_GD,
alvarez2016fixed,maunu2023bures,PanaretosSantoroBW}.
This can be seen as a special case of the literature on Wasserstein barycenters \cite{WassersteinBarycenter,WassersteinConvergence,AlvarezEstebanAggregate} or more generally of the literature on Fr\'echet means in abstract metric spaces \cite{EvansJaffeSLLN,schotz2022strong,sturm2003probability, ziezold1977expected, Sverdrup, Brunel, AGP}.
In the finite-dimensional setting, our results can be seen as a special case of the literature of barycenters on Riemannian manifolds \cite{Bhattacharya2003, Bhattacharya2005, kendall2011limit} or more generally of the literature of barycenters in stratified spaces \cite{Huckemann, BookCLTs, UnlabeledGraphs}.

We emphasize that our main result is novel in all of the above contexts.
Indeed, known results on Fr\'echet means or barycenters of IID samples include laws of large numbers \cite{ziezold1977expected, EvansJaffeSLLN, Sverdrup, schotz2022strong, Bhattacharya2003, PanaretosSantoroBW, UnlabeledGraphs}, central limit theorems \cite{agueh2017vers, Bhattacharya2003,  Bhattacharya2005, PanaretosSantoroBW, BookCLTs, Huckemann, FrechetCLTs}, concentration inequalities \cite{Kroshnin, Brunel}, rates of convergence \cite{WassersteinConvergence, SchoetzQuad, AGP} and more, but large deviations theory has been apparently overlooked so far.
The only result we are aware of is \cite[Theorem~4.6]{EvansJaffeSLLN} by the first author, which provides large deviations upper bounds for Fr\'echet means on compact metric spaces via straightforward continuity arguments; however, the form of the rate function therein is too complicated to yield any meaningful probabilistic or geometric insights.
(One also has the seemingly-related work \cite{kraaij2019classical}, but its focus is large deviations theory for suitably-scaled random walks rather than for barycenters.)
As a general comment, it is arguably surprising that no large deviations results were known even in the simplest case of, say, barycenters on finite-dimensional Riemannian manifolds of non-positive curvature.

Additionally, our method of proof distinguishes itself from existing methods.
To explain this, let us recall the case of barycenters in finite-dimensional Riemannian manifolds and barycenters in the Wasserstein and Bures-Wasserstein spaces, in which case one has a suitable notion of ``linearization'' around each point; in these settings, most results \cite{WassersteinConvergence, Kroshnin,PanaretosSantoroBW, Brunel, AGP, Bhattacharya2003, Bhattacharya2005, kendall2011limit, BookCLTs, Huckemann, FrechetCLTs, carlier2021entropic} are proven by lifting the empirical barycenter to the ``tangent space'' at the population barycenter, so that one can apply standard probabilistic techniques in a fixed Euclidean setting.
While this strategy makes sense for analysis of the \textit{typical} behavior of the empirical barycenter, it does not make sense for analyses of its \textit{atypical} behavior.
Necessarily, our probabilistic results implicitly concern all tangent spaces simultaneously.

Lastly, our results may be interesting from the point of view of large deviations theory itself.
Our main theorem is an example of the notoriously difficult situation where one seeks a concrete and interpretable large deviations theory for a non-linear and possibly discontinuous function of IID random variables.
Indeed, while large deviations theory for linear functions of IID Euclidean random variables is an extensively studied topic, there are no general tools for large deviations of nonlinear functionals (see \cite[Section~1.2: ``The Problem with Non-Linearity'']{chatterjee2017large}), especially in infinite-dimensional spaces.

\section{Notation and Preliminaries}\label{sec:prelim}

Let $\Hilbert$ be some real separable Hilbert space with $\mathrm{dim}(\Hilbert) \in \Nbb\cup\{\infty\}$. Denote the set of real, self-adjoint operators by $\SymSpace$, and the Hilbert-Schmidt norm on $\SymSpace$ by
 $$
 \|A\|_{2}:= \sqrt{\trace(A^{\top}A)},
 $$
 which is induced by the inner product
 $$
\langle A,A'\rangle_{2} := \trace(A^{\top}A').
$$
Also denote the trace norm by
 $$
 \|A\|_{1}:= \trace(\sqrt{A^{\top}A}).
 $$
 We say that $A\in\SymSpace$ is \textit{Hilbert-Schmidt} if $\|A\|_2 < \infty$ and \textit{trace class} if $\|A\|_1<\infty$.
For $A\in\SymSpace$ we write $A\succeq 0$ and say that $A$ is \textit{non-negative} if $\langle h ,Ah \rangle \geq 0$ for all vectors $h\in\Hilbert$, and we write $A\succ 0$ and say that $A$ is \textit{positive} (or \textit{regular}) if the inequality is strict for all $h\neq 0$.

\medskip

We write $\CovSpace$ for the set of real, self-adjoint, non-negative, trace class operators on $\Hilbert$ (that is, the space of covariances on $\Hilbert$), and we write $\CovSpacePos\subseteq\CovSpace$ for the subspace of positive covariances.
Note that, if $d:=\dim(\Hilbert)<\infty$, then $\CovSpace$ is just the set of real, symmetric, positive semi-definite $d\times d$ matrices, and $\CovSpacePos\subseteq\CovSpace$ is the subspace of positive definite matrices.
For $M\in\CovSpace$, we write $M^{\sfrac{1}{2}}$ for the operator-theoretic square root, meaning if $M$ has spectral decomposition $U\Lambda U^{\top}$ then $M^{\sfrac{1}{2}} := U\Lambda^{\sfrac{1}{2}}U^{\top}$.

\medskip

The Bures-Wasserstein metric $\BW$ between covariances $\Sigma,\Sigma'\in \CovSpace$, defined in \eqref{eqn:BWdef}, may be  equivalently represented as the 2-Wasserstein distance between the corresponding centered Gaussian measures, $\Ncal(0,\Sigma), \Ncal(0,\Sigma')$, respectively.
(We direct the reader to \cite{bogachev1998gaussian} for general theory of Gaussian measures in infinite-dimensions.)
This connection to optimal transport yields an important perspective which will later allow for some geometric interpretations of our results: indeed, it is known that the Wasserstein space admits a manifold-like structure formulated in terms of optimal transport maps (the so-called Otto calculus \cite{otto2001geometry}), which is in turn inherited by the Bures-Wasserstein space. In more detail,
for $M\in \CovSpacePos$,  the space:
\begin{equation}
    \label{eqn:tanspace}
\SymSpace_M := \{ A \in \SymSpace\::\: \|M^{\sfrac{1}{2}}A\|_2<\infty\}
\end{equation}
 is a Hilbert space, if endowed with the inner product:
$$
\langle A, B\rangle_M = \trace(AMB),
$$
and the pair $(\SymSpace_M,\langle\,\cdot\,,\,\cdot\,\rangle_M)$ is sometimes formally referred to as the \textit{tangent space at $M$}; see \cite{Masarotto, takatsu2010wasserstein}. In fact, the correspondence:
\begin{equation}\label{eqn:log}
    \ell_M\::\: 
    (\CovSpace,\BW) \to (\SymSpace_M,\langle\,\cdot\,,\,\cdot\,\rangle_M),\qquad \Sigma \mapsto t_{M}^{\Sigma} - \Id
\end{equation}
 is an embedding of $\CovSpace$ into the Hilbert space $\SymSpace_M$, for any $M\succ 0$.
Here, the operator $ t_{M}^{\Sigma}$ denotes the \textit{optimal transport map} between $\Ncal(0,M)$ and $\Ncal(0,\Sigma)$; see \cite{Villani}. In our context, this can be expressed as
\begin{equation}
    \label{eqn:optmap}
    t_{M}^{\Sigma} := M^{-\sfrac{1}{2}}(M^{\sfrac{1}{2}}\Sigma M^{\sfrac{1}{2}})^{\sfrac{1}{2}}{M}^{-\sfrac{1}{2}},
\end{equation}
which is well-defined on a subspace of $\Hilbert$ which simultaneously contains the range of $M^{\sfrac{1}{2}}$ and has full probability under $\Ncal(0,M)$; however, the operator $t_{M}^{\Sigma}$ can generally be unbounded \cite[Proposition 2.2]{CuestaAlbertos}. 
One also has the following equivalent expression for $\BW$
\begin{equation}
    \label{eqn:BWequivMaps}
\BW(M,\Sigma) = \|(t_M^{\Sigma}-\Id) \|_{M},
\end{equation}
provided that an optimal map exists.
 
The embedding \eqref{eqn:log} admits an inverse given by
\begin{equation}\label{eqn:exp}
    e_M\::\: (\SymSpace_M,\langle\,\cdot\,,\,\cdot\,\rangle_M)\to (\CovSpace,\BW),\qquad A \mapsto (A+\Id)M(A+\Id).
\end{equation}
In some contexts, these maps are referred to as the \textit{logarithm} map and \textit{exponential} map, respectively, which highlights the manifold-like structure of the Bures-Wasserstein space.
See \cite{Masarotto} for further detail.

\medskip

Next we develop some analytic preliminaries that will be used in the proofs of our main results.
To set this up, let us write $\Pcal(\CovSpace)$ for the space of Borel probability measures on $(\CovSpace,\BW)$.
Then, for $p\ge 1$, we write $\Pcal_p(\CovSpace)\subseteq \Pcal(\CovSpace)$ for the space of all probability measures satisfying $\int_{\CovSpace}\BW^p(\Sigma,0)\diff P(\Sigma) < \infty$, and we write $W_p$ for the $p$-Wasserstein metric on $\Pcal_p(\CovSpace)$.
We also write $\weak$ for the topology of weak convergence on $\Pcal(\CovSpace)$.

\medskip

The last step is to precisely define the barycenter set under the appropriate moment condition.
Recall that in equation~\eqref{eqn:pop-bary} from the introduction, we considered, for each $P\in\Pcal_2(\CovSpace)$, the set of minimizers
\begin{equation}\label{eqn:barycenter-2nd-mom}
    \underset{M\in\CovSpace}{\textnormal{arg min}}\int_{\CovSpace}\BW^2(M,\Sigma)\diff P(\Sigma),
\end{equation}
which can be empty, a singleton, or contain more than one point.
While problem~\eqref{eqn:barycenter-2nd-mom} makes sense for any $P\in\Pcal_2(\CovSpace)$, it turns out that the later results in the paper will require the weaker moment assumption $P\in\Pcal_1(\CovSpace)$.
To see how to define the barycenter in this case, we define, for each $P\in\Pcal_1(\CovSpace)$,
\begin{equation*}
    F(P):=\underset{M\in\CovSpace}{\textnormal{arg min}}\int_{\CovSpace}(\BW^2(M,\Sigma)-\trace(\Sigma))\diff P(\Sigma),
\end{equation*}
which again is the set of minimizers.
This is well-defined under only $P\in\Pcal_1(\CovSpace)$, since we have $\trace(\Sigma) = \BW^2(0,\Sigma)$ where $0\in\CovSpace$ is the zero operator, and, as in any metric space, we have the inequality
\begin{equation*}
    |\BW^2(M,\Sigma)-\BW^2(M_0,\Sigma)|\le \BW(M,M_0)(\BW(M,\Sigma)+\BW(M_0,\Sigma))
\end{equation*}
for all $M,M_0,\Sigma\in\CovSpace$.
Additionally, if $P\in\Pcal_2(\CovSpace)$, then $\trace(\Sigma)$ is integrable hence $F(P)$ coincides with the set \eqref{eqn:barycenter-2nd-mom}.

Lastly, we need the following result which shows that the map $F$ is continuous on the domain $(\Pcal_1(\CovSpace),W_1)$; this extends the recent results \cite[Remark~4.1 and Theorem~4.2]{PanaretosSantoroBW} by the second author which prove the same result but for the smaller domain $(\Pcal_2(\CovSpace),W_2)$.
As we will see in Remark~\ref{rem:W1-v-W2}, this extension is necessary for establishing the LDP under the minimal moment condition possible.

\begin{theorem}\label{thm:cty-general}
    Suppose $\{P_n\}_{n\in\Nbb}$ and $P$ in $\Pcal_1(\CovSpace)$ satisfy $W_1(P_n,P)\to 0$.
	Then, for any $\{M_n\}_{n\in\Nbb}$ in $\CovSpace$ with $M_n\in F(P_n)$ for all $n\in\Nbb$, the sequence $\{M_n\}_{n\in\Nbb}$ is relatively compact and any subsequential limit is contained in $F(P)$.
\end{theorem}

\begin{proof}
    Recall by $W_1(P_n,P)\to 0$  and \cite[Proposition~7.1.5]{ambrosio2004gradient} that $\{P_n\}_{n\in\Nbb}$ is tight and makes the function $\Sigma\mapsto\BW(\Sigma,0)$ uniformly integrable; as such, it is possible to choose $L>0$ large enough so that we have both
    \begin{equation}\label{eqn:probab-bounds}
        P_n(\CovSpace\setminus B_L(0))\le \frac{1}{8}\qquad\textnormal{ and }\qquad \int_{\CovSpace\setminus B_L(0)}\BW(\Sigma,0) \diff P_n(\Sigma)\le 1
    \end{equation}
    for all $n\in\Nbb$.
    Second, we develop some estimates.
    For one, use
    the elementary inequality
    \begin{equation*}
        \BW^2(M,0) \le 2(\BW^2(\Sigma,0) + \BW^2(M,\Sigma)) = 2(\trace(\Sigma) + \BW^2(M,\Sigma))
    \end{equation*}
    to obtain
    
    \begin{equation}\label{eqn:estimate-1}
        \BW^2(M,\Sigma) - \trace(\Sigma) \ge \frac{1}{2}\BW^2(M,0) - 2\trace(\Sigma).
    \end{equation}
    For another, use the triangle inequality twice to get
    \begin{align*}
        |\BW^2(M,\Sigma)-\trace(\Sigma)| &= |\BW^2(M,\Sigma)-\BW^2(\Sigma,0)|\\
        &\le \BW(M,0)(\BW(M,\Sigma)+\BW(0,\Sigma)) \\
        &\le \BW(M,0)(\BW(M,0)+2\BW(0,\Sigma))
    \end{align*}
    hence
    \begin{equation}\label{eqn:estimate-2}
        \BW^2(M,\Sigma)-\trace(\Sigma)\ge -\BW^2(M,0) - 2\BW(M,0)\BW(\Sigma,0).
    \end{equation}
    We will use these bounds throughout the proof.

    Next, we claim that $\BW(M_n,0)\le R$ for all $n\in\Nbb$, where
    \begin{equation*}
        R= 4\left(2+\sqrt{4+\frac{L^2}{2}}\right).
    \end{equation*}
    To see this, assume for the sake of contradiction that some $n\in\Nbb$ satisfies $\BW(M_n,0)>R$.
    Then use \eqref{eqn:probab-bounds} and \eqref{eqn:estimate-1} to get
    \begin{align*}
        \int_{B_L(0)}(\BW^2(M_n,\Sigma)-\trace(\Sigma))\diff P_n(\Sigma) &\ge \int_{B_L(0)}\left(\frac{1}{2}\BW^2(M_n,0)-2\trace(\Sigma)\right)\diff P_n(\Sigma) \\
        &\ge \int_{B_L(0)}\left(\frac{1}{2}\BW^2(M_n,0)-2L^2\right)\diff P_n(\Sigma) \\
        &\ge \left(\frac{1}{2}\BW^2(M_n,0)-2L^2\right) P_n(B_L(0)) \\
        &\ge \frac{1}{4}\BW^2(M_n,0)-L^2.
    \end{align*}
    Similarly, use \eqref{eqn:probab-bounds} and \eqref{eqn:estimate-2} to get
    \begin{align*}
        \int_{\CovSpace\setminus B_L(0)}&(\BW^2(M_n,\Sigma)-\trace(\Sigma))\diff P_n(\Sigma) \\
        &\ge -\BW^2(M_n,0)P_n(\CovSpace\setminus B_L(0)) -2\BW(M_n,0)\int_{\CovSpace\setminus B_L(0)}\BW(\Sigma,0)\diff P_n(\Sigma) \\
        &\ge -\frac{1}{8}\BW^2(M_n,0)-2\BW(M_n,0).
    \end{align*}
    Adding together the previous two displays shows
    \begin{equation*}
        \int_{\CovSpace}(\BW^2(M_n,\Sigma)-\trace(\Sigma))\diff P_n(\Sigma) \ge \frac{1}{8}\BW^2(M_n,0) - 2\BW(M_n,0) - L^2.
    \end{equation*}
    Since $M_n\in F(P_n)$ implies that the left side is non-positive, we have shown $p(\BW(M_n,0))\le 0$ where $p$ is the polynomial $p(u) = \frac{1}{8}u^2-2u - L^2$, and this implies $\BW(M_n,0) \le 4(2+\sqrt{4+2^{-1}L^2})$ as claimed.

    Observe that, writing $\mu_n=\Ncal(0,M_n)$ for all $n\in\Nbb$, the previous paragraph establishes
    \begin{equation*}
        \sup_{n\in\Nbb}W_2(\mu_n,\delta_0)=\sup_{n\in\Nbb}\BW(M_n,0)\le R < \infty.
    \end{equation*}
    Thus, for all $r\ge 0$ we have, by Markov's inequality:
    \begin{equation*}
        \sup_{n\in\Nbb}\mu_n(\Hilbert\setminus B_r(0)) \le \frac{R^2}{r^2}
    \end{equation*}
    We want to combine this bound with Prokhorov's theorem in order to extract a suitable convergent subsequence of $\{\mu_n\}_{n\in\Nbb}$, but this requires us to introduce a metrizable topology on $\Hilbert$ weak enough so that $\|\cdot\|$-bounded sets are relatively compact.
    We do this by following the construction \cite[Section~5.1.2]{ambrosio2004gradient}:
    Let $\{e_i\}_{i\in\Nbb}$ be a complete orthonormal system (CONS) for $\Hilbert$, and write $\|\cdot\|'$ for the pre-Hilbert norm on $\Hilbert$ defined via $\|h\|' := (\sum_{i\in\Nbb}i^{-2}\langle h,e_i\rangle^2)^{1/2}$; then $(\Hilbert,\|\cdot\|')$ is a pre-Hilbert space, and $\|\cdot\|$-bounded sets are relatively $\|\cdot\|'$-compact.
    (In fact, the weak topology of $(\Hilbert,\|\cdot\|)$ coincides with the topology of $(\Hilbert,\|\cdot\|')$ when restricted to bounded sets.)
    Therefore, by \cite[Theorem~5.1.12(c)]{ambrosio2004gradient} there exists a subsequence $\{n_k\}_{k\in\Nbb}$ and a probability measure $\mu\in\Pcal(\Hilbert)$ such that we have $\int_{\Hilbert}\phi\diff \Ncal(0,M_{n_k})\to\int_{\Hilbert}\phi\diff \mu$ for all bounded continuous $\phi:(\Hilbert,\|\cdot\|')\to\Rbb$; moreover \cite[Theorem~5.1.12(f)]{ambrosio2004gradient} guarantees that all finite dimensional distributions of $\mu$ are Gaussian hence $\mu$ is Gaussian, i.e., there exists some $M\in\CovSpace$ such that $\mu = \Ncal(0,M)$.
    In summary, we have
    \begin{equation}\label{eqn:cty-def-weak}
        \int_{\Hilbert}\phi\diff \Ncal(0,M_{n_k})\to\int_{\Hilbert}\phi\diff \Ncal(0,M)
    \end{equation}
    for all bounded continuous $\phi:(\Hilbert,\|\cdot\|')\to\Rbb$.

    Next we claim that $M\in F(P)$.
    To do this, use \eqref{eqn:cty-def-weak} and \cite[Lemma~7.1.4]{ambrosio2004gradient} to get
    \begin{equation*}
        \BW^2(M,\Sigma)\le \liminf_{k\to\infty}\BW^2(M_{n_k},\Sigma)
    \end{equation*}
    for all $\Sigma\in\CovSpace$.
    Also, we claim that any bounded sequence $\{M_k'\}_{k\in\Nbb}$ in $\CovSpace$ satisfies
    \begin{equation}\label{eqn:cty-lemma}
        \liminf_{k\to\infty}\int_{\CovSpace}(\BW^2(M_k',\Sigma)-\trace(\Sigma))\diff P_{n_k}(\Sigma) = \liminf_{k\to\infty}\int_{\CovSpace}(\BW^2(M_k',\Sigma)-\trace(\Sigma))\diff P(\Sigma).
    \end{equation}
    To see this, suppose $\sup_{k\in\Nbb}\BW(M_k',0)\le R$, and use Skorokhod's representation theorem to construct a probability space on which $\Sigma_k$ has distribution $P_{n_k}$ for all $k\in\Nbb$, $\Sigma$ has distribution $P$, and $\BW(\Sigma_k,\Sigma)\to 0$ holds almost surely; then use the triangle inequality to bound
    \begin{align*}
        &|\BW^2(M_k',\Sigma_k) -\trace(\Sigma_k)-(\BW^2(M_k',\Sigma)-\trace(\Sigma))| \\
        &\le \BW(\Sigma_k,\Sigma)(\BW(M_k',\Sigma_k)+\BW(M_k',\Sigma)+\BW(\Sigma_k,0)+\BW(\Sigma,0)) \\
        &\le 2\BW(\Sigma_k,\Sigma)(R+\BW(\Sigma_k,0)+\BW(\Sigma,0)),
    \end{align*}
    and note that the right side converges to zero in expectation, since it is uniformly integrable and converges to zero almost surely.
    Lastly, note that \eqref{eqn:estimate-2} shows that the functions $\Sigma\mapsto\BW^2(M_{n_k},\Sigma)-\trace(\Sigma)$ have a uniform lower bound which is integrable with respect to $P$, so we may apply Fatou's lemma.
    Combining these observations shows that for any $M'\in\CovSpace$, we have
    \begin{equation}\label{eqn:Fp-containment}
        \begin{split}
        \int_{\CovSpace}(\BW^2(M,\Sigma)-\trace(\Sigma))\diff P(\Sigma) &\le \int_{\CovSpace}\liminf_{k\to\infty}(\BW^2(M_{n_k},\Sigma)-\trace(\Sigma))\diff P(\Sigma) \\
        &\le \liminf_{k\to\infty}\int_{\CovSpace}(\BW^2(M_{n_k},\Sigma)-\trace(\Sigma))\diff P(\Sigma) \\
        &= \liminf_{k\to\infty}\int_{\CovSpace}(\BW^2(M_{n_k},\Sigma)-\trace(\Sigma))\diff P_{n_k}(\Sigma) \\
        &\le \liminf_{k\to\infty}\int_{\CovSpace}(\BW^2(M',\Sigma)-\trace(\Sigma))\diff P_{n_k}(\Sigma) \\
        &= \int_{\CovSpace}(\BW^2(M',\Sigma)-\trace(\Sigma))\diff P(\Sigma).
    \end{split}
    \end{equation}
    Since $M'\in\CovSpace$ was arbitrary, this establishes $M\in F(P)$ as claimed.

    Lastly, we claim that $\BW(M_{n_k},M)\to 0$.
    To show this, define the set $A_K\subseteq \CovSpace$, for any subsequence $K = \{k_j\}_{j\in\Nbb}$, as follows:
	\begin{equation*}
		A_K := \left\{\Sigma\in \CovSpace: \liminf_{j\to\infty} \BW(M_{n_{k_j}},\Sigma) = \BW(M,\Sigma) \right\}.
	\end{equation*}
    Because $M\in F(P)$ implies that the inequalities in \eqref{eqn:Fp-containment} are in fact equalities, we have $P(A_K) = 1$ for all $K$.
    But we also observe that $A_K$ is a closed subset of $(\CovSpace,\BW)$; if $\{\Sigma_{\ell}\}_{\ell\in\Nbb}$ in $A_K$ satisfy $\BW(\Sigma_{\ell},\Sigma)\to 0$, then    
    \begin{align*}
		&\left|\liminf_{j\to\infty}\BW(M_{n_{k_j}},\Sigma)-\BW(M,\Sigma)\right|  \\
		&= 		\left|\liminf_{j\to\infty}\BW(M_{n_{k_j}},\Sigma)-\liminf_{j\to\infty}\BW(M_{n_{k_j}},\Sigma_\ell)+\BW(M,\Sigma_\ell)-\BW(M,\Sigma)\right|  \\
		&\le	\limsup_{j\to\infty}\left|\BW(M_{n_{k_j}},\Sigma)-\BW(M_{n_{k_j}},\Sigma_\ell)\right|+\left|\BW(M,\Sigma_\ell)-\BW(M,\Sigma)\right|  \\
		&\le 2\BW(\Sigma_\ell, \Sigma)
	\end{align*}
    which goes to zero as $\ell\to\infty$, hence $\Sigma\in A_K$.
    Since an arbitrary intersection of full measure closed sets is full measure in a separable space, this implies $A:=\bigcap_{K}A_K$ satisfies $P(A) = 1$, where the intersection runs over all subsequences $K$.
    In particular, $A$ is non-empty, so we can choose some $\Sigma_0\in A$ and it follows that $\BW(M_{n_k},\Sigma_0)\to \BW(M,\Sigma_0)$.
    Combining this with \eqref{eqn:cty-def-weak}, we obtain from \cite[Lemma~14]{le2017existence} that $\BW(M_{n_k},M)\to 0$.
    This finishes the proof.
\end{proof}

The preceding result did not make any assumptions about the uniqueness of the barycenter, but it is known \cite{alvarez2016fixed} that $F(P)$ is unique whenever $P\{\Sigma:\Sigma\succ 0\} = 1$; in this case, the preceding result has the following simpler statement.

\begin{corollary}\label{cor:cty-uniqueness}
    The function $F:(\Pcal_1(\CovSpacePos),W_1)\to (\CovSpace,\BW)$ is continuous.
\end{corollary}

\section{Main Results}\label{sec:main}

In this section we will state and prove the main results of the paper.
We begin by enumerating and discussing our assumptions.
Then, in Subsection~\ref{subsec:exp-tilt} we develop a fundamental notion of exponential tilting for probability measures on $(\CovSpace,\BW)$, and in Subsection~\ref{subsec:LDP} we use this to establish our main results.
Next, we study some consequences of the main results in Subsection~\ref{subsec:consequences} which focuses mainly on the phenomenon of concentration of measure for Bures-Wasserstein barycenters.
Lastly, we give some comments on numerical implementation of minimization of the rate function (and some related numerical experiments) in Subsection~\ref{subsec:numerical}.

\medskip

Our first assumption is as follows:
\begin{itemize}
    \item[\namedlabel{int}{\textbf{(E)}}] For all $\lambda > 0$ we have  $\int_{\CovSpace}\exp(\lambda \BW(\Sigma,0))\,\textnormal{d}P(\Sigma) < \infty$.
\end{itemize}

Condition~\textnormal{\hyperref[int]{\textbf{(E)}}} simply states that, if $\Sigma$ is distributed according to $P$, then the real-valued random variable $\BW(\Sigma,0)$ has all exponential moments finite.
This condition is weaker than requiring that the distribution of $\BW(\Sigma,0)$ is sub-Gaussian (as in \cite[Assumption~2]{Kroshnin}) or has bounded support (as in \cite[Theorem~1]{BW_GD}), which have been assumed in previous literature focused only on the finite-dimensional setting.
Of course, condition~\textnormal{\hyperref[int]{\textbf{(E)}}} is stricter than first- and second-moment conditions which are used in proofs of the SLLN \cite{schotz2022strong,JaffeInfDim} and CLT \cite{Kroshnin,PanaretosSantoroBW} respectively, but this is to be expected (from the point of view of large deviations theory) and aligns with the classical (Euclidean) setting.

Second, we need the following:

\begin{itemize}
        \item[\namedlabel{reg}{\textbf{(R)}}]         There exist measurable $c,C:\CovSpace\to (0,\infty)$ and a fixed covariance $R\in \CovSpacePos$ such that we have $P\{\Sigma: c(\Sigma)R \preceq \Sigma \preceq C(\Sigma)R\} = 1$ and $\int_{\CovSpace}C(\Sigma)\diff P(\Sigma) < \infty$.
\end{itemize}

The assumption~\textnormal{\hyperref[reg]{\textbf{(R)}}} states that essentially all $\Sigma$ generate the same reproducing kernel Hilbert space (RKHS), i.e., they are almost surely equivalent to some $R\in\CovSpacePos$, and (roughly speaking) that the ratio of their corresponding eigenvalues is integrable.
This is primarily an assumption on the regularity of a sample $\Sigma$ from $P$, although it also contains some small assumption about integrability.
We note that this assumption appears in \cite{PanaretosSantoroBW} to obtain a rate of convergence and a CLT; in fact, it is known \cite{PanaretosSantoroBW} that, in the finite-dimensional case, it is equivalent to
\begin{equation}\label{eqn:finite-dim-regularity}
        \int_{\CovSpace}\BW^2(\Sigma,0)\diff P(\Sigma) <\infty \qquad\textnormal{ and }\qquad P\{\Sigma: \Sigma\succ 0\} = 1,
    \end{equation}
    and is hence similar to the one required in \cite[Assumption~1]{Kroshnin}.

Since assumption~\textnormal{\hyperref[reg]{\textbf{(R)}}} always implies $P\{\Sigma:\Sigma\succ 0\} = 1$, it follows that $P$ must have a unique barycenter \cite{alvarez2016fixed}.
Furthermore, it is known \cite{PanaretosSantoroBW} that, in the presence of  assumption~\textnormal{\hyperref[reg]{\textbf{(R)}}}, the unique Bures-Wasserstein barycenter is \textit{characterized} by a fixed-point equation. Indeed, if $P\in\Pcal_2(\CovSpace)$ satisfies assumption~\textnormal{\hyperref[reg]{\textbf{(R)}}}, then $P$ has a unique barycenter, and $M\in\CovSpace$ is the barycenter of $P$ if and only if
\begin{equation}\label{eqn:pop-fixed-pt}
	M\succ 0 \qquad \text{and}\qquad \int_{\CovSpace}(t_{M}^{\Sigma}-\Id)\diff P(\Sigma) = 0,
\end{equation}
where the integral is a Bochner integral in the Hilbert space $\SymSpace_M$ \citep[][Definitions 2.6.1--3]{hsing2015theoretical}.

To gain some intuition for the fixed-point equation \eqref{eqn:pop-fixed-pt}, note that it essentially states that a covariance $M$ is the barycenter of $P$ if and only if it is, on average, preserved by the random displacement $t_{M}^{\Sigma}$.
It also leads us to a geometric interpretation in terms of the Bures-Wasserstein space: a covariance $M$ is the barycenter of $P$ if and only if the pushforward of $P$ by the embedding $\ell_M$ defined in \eqref{eqn:log} is centered in $\SymSpace_M$.
See Figure~\ref{fig:fixed-pt-equation} for an illustration.

\begin{figure}[t]
	\begin{center}
		\begin{tikzpicture}[scale=1.7]
			
			\coordinate (A) at (-0.2, 0.7);
			\coordinate (B) at (-1.85, 0.3);
			\coordinate (C) at (-2.7, 0.95);
			\coordinate (D) at (-0.9, 1.4);
			\fill[thin, opacity=0.75, color=gray!30] (A) to[out=170,in=40] (B) to[out=120,in=350] (C) to[out=50,in=170] (D) to[out=330,in=120] (A);
			\draw[thin, solid] (A) to[out=170,in=40] (B) to[out=120,in=350] (C) to[out=50,in=170] (D);
			\draw[thin, dashed] (D) to[out=330, in=120] (A);
			\node at (-0.5, 0.5) {$\mathbb{K}$};
			
			\coordinate (D1) at (-1.8, 1.1);
			\coordinate (D2) at (-1.2, 1.25);
			\coordinate (D3) at (-1.85, 0.5);
			\coordinate (D4) at (-1.7, 1.3);
			\coordinate (D5) at (-1.75, 0.7);
			\coordinate (D6) at (-1.9, 0.72);
			\coordinate (D7) at (-1.55, 0.9);
			
			\filldraw[blue!50] (D1) circle (0.02) node[] {};
			\filldraw[blue!50] (D2) circle (0.02) node[] {};
			\filldraw[blue!50] (D3) circle (0.02) node[] {};
			\filldraw[blue!50] (D4) circle (0.02) node[] {};
			\filldraw[blue!50] (D5) circle (0.02) node[] {};
			\filldraw[blue!50] (D6) circle (0.02) node[] {};
			\filldraw[blue!50] (D7) circle (0.02) node[] {};
			\node[color=blue] at (-2.0, 1.65) {$P$};
		\end{tikzpicture} \begin{tikzpicture}[scale=1.7]
			
			\coordinate (A) at (-0.2, 0.7);
			\coordinate (B) at (-1.85, 0.3);
			\coordinate (C) at (-2.7, 0.95);
			\coordinate (D) at (-0.9, 1.4);
			\fill[thin, opacity=0.75, color=gray!30] (A) to[out=170,in=40] (B) to[out=120,in=350] (C) to[out=50,in=170] (D) to[out=330,in=120] (A);
			\draw[thin, solid] (A) to[out=170,in=40] (B) to[out=120,in=350] (C) to[out=50,in=170] (D);
			\draw[thin, dashed] (D) to[out=330, in=120] (A);
			\node at (-0.5, 0.5) {$\mathbb{K}$};
			
			\coordinate (D1) at (-1.8, 1.1);
			\coordinate (D2) at (-1.2, 1.25);
			\coordinate (D3) at (-1.85, 0.5);
			\coordinate (D4) at (-1.7, 1.3);
			\coordinate (D5) at (-1.75, 0.7);
			\coordinate (D6) at (-1.9, 0.72);
			\coordinate (D7) at (-1.55, 0.9);

			\draw[thick, dash pattern=on 1pt off 1pt, color=blue] (D1) to (-1.8,1.2);
			\draw[thick, dash pattern=on 1pt off 1pt, color=blue] (D2) to (-1.2,1.35);
			\draw[thick, dash pattern=on 1pt off 1pt, color=blue] (D3) to (-1.85,0.9);
			\draw[thick, dash pattern=on 1pt off 1pt, color=blue] (D4) to (-1.7,1.35);
			\draw[thick, dash pattern=on 1pt off 1pt, color=blue] (D5) to (-1.75,1.0);
			\draw[thick, dash pattern=on 1pt off 1pt, color=blue] (D6) to (-1.9,1.0);
			\draw[thick, dash pattern=on 1pt off 1pt, color=blue] (D7) to (-1.55,1.1);

			\filldraw[thin, opacity=0.75, color=red!30, xscale=-0.5, yscale=0.25,shift={(1.75,4.25)}] (0, 0) to (2, -1) to (3,1) to (1, 2) to cycle;
			\draw[thin, color=red, xscale=-0.5, yscale=0.25,shift={(1.75,4.25)}] (0, 0) to (2, -1) to (3,1) to (1, 2) to cycle;
			\filldraw[red] (-1.6, 1.2) circle (0.02) node[above right] {$M$};
			\node[color=red] at (-2.0, 1.65) {$\SymSpace_{M}$};
			
			\filldraw[blue!50] (-1.8,1.2) circle (0.02) node[] {};
			\filldraw[blue!50] (-1.21,1.35) circle (0.02) node[] {};
			\filldraw[blue!50] (-1.85,0.9) circle (0.02) node[] {};
			\filldraw[blue!50] (-1.7,1.35) circle (0.02) node[] {};
			\filldraw[blue!50] (-1.75,1.0) circle (0.02) node[] {};
			\filldraw[blue!50] (-1.9,1.0) circle (0.02) node[] {};
			\filldraw[blue!50] (-1.55,1.1) circle (0.02) node[] {};
		\end{tikzpicture} \begin{tikzpicture}[scale=1.7]
			
			\coordinate (A) at (-0.2, 0.7);
			\coordinate (B) at (-1.85, 0.3);
			\coordinate (C) at (-2.7, 0.95);
			\coordinate (D) at (-0.9, 1.4);
			\fill[thin, opacity=0.75, color=gray!30] (A) to[out=170,in=40] (B) to[out=120,in=350] (C) to[out=50,in=170] (D) to[out=330,in=120] (A);
			\draw[thin, solid] (A) to[out=170,in=40] (B) to[out=120,in=350] (C) to[out=50,in=170] (D);
			\draw[thin, dashed] (D) to[out=330, in=120] (A);
			\node at (-0.5, 0.5) {$\mathbb{K}$};
			
			\coordinate (D1) at (-1.8, 1.1);
			\coordinate (D2) at (-1.2, 1.25);
			\coordinate (D3) at (-1.85, 0.5);
			\coordinate (D4) at (-1.7, 1.3);
			\coordinate (D5) at (-1.75, 0.7);
			\coordinate (D6) at (-1.9, 0.72);
			\coordinate (D7) at (-1.55, 0.9);

			\draw[thick, dash pattern=on 1pt off 1pt, color=blue] (D1) to ($(D1) + (-0.0125,0.05)$);
			\draw[thick, dash pattern=on 1pt off 1pt, color=blue] (D2) to ($(D2) + (-0.0375,0.15)$);
			\draw[thick, dash pattern=on 1pt off 1pt, color=blue] (D3) to ($(D3) + (-0.0375,0.15)$);
			\draw[thick, dash pattern=on 1pt off 1pt, color=blue] (D4) to ($(D4) + (-0.025,0.1)$);
			\draw[thick, dash pattern=on 1pt off 1pt, color=blue] (D5) to ($(D5) + (-0.025,0.1)$);
			\draw[thick, dash pattern=on 1pt off 1pt, color=blue] (D6) to ($(D6) + (-0.025,0.1)$);
			\draw[thick, dash pattern=on 1pt off 1pt, color=blue] (D7) to ($(D7) + (-0.0125,0.05)$);
			
			\definecolor{Green}{HTML}{27912b}
			\filldraw[thin, opacity=0.75, color=Green!30, xscale=-0.5, yscale=0.4,shift={(1.65,2.1)}] (0, 0) to (2, -1) to (3,1) to (1, 2) to cycle;
			\draw[thin, color=Green, xscale=-0.5, yscale=0.4,shift={(1.65,2.1)}] (0, 0) to (2, -1) to (3,1) to (1, 2) to cycle;
			\filldraw[Green] (-1.6, 1.05) circle (0.02) node[above right] {$\bary$};
			\node[color=Green] at (-2.0, 1.65) {$\SymSpace_{\bary}$};
			
			\filldraw[blue!50] ($(D1) + (-0.0125,0.05)$) circle (0.02) node[] {};
			\filldraw[blue!50] ($(D2) + (-0.0375,0.15)$) circle (0.02) node[] {};
			\filldraw[blue!50] ($(D3) + (-0.0375,0.15)$) circle (0.02) node[] {};
			\filldraw[blue!50] ($(D4) + (-0.025,0.1)$) circle (0.02) node[] {};
			\filldraw[blue!50] ($(D5) + (-0.025,0.1)$) circle (0.02) node[] {};
			\filldraw[blue!50] ($(D6) + (-0.025,0.1)$) circle (0.02) node[] {};
			\filldraw[blue!50] ($(D7) + (-0.0125,0.05)$) circle (0.02) node[] {};
		\end{tikzpicture}
	\end{center}
	\caption{The fixed-point equation for the Bures-Wasserstein barycenter.
    For a probability measure $P$ on $(\CovSpace,\BW)$ (left) and for each covariance $M\succ 0$, we can push forward $P$ by $\ell_M$ to a probability measure on $(\SymSpace_M,\langle\,\cdot\,,\,\cdot\,\rangle_M)$.
    This pushforward is centered if and only if $M$ is a Bures-Wasserstein barycenter of $P$ (middle, right).}
    \label{fig:fixed-pt-equation}
\end{figure}

\medskip

For the remainder of this section we will assume that $P\in\Pcal(\CovSpace)$ is fixed and satisfies \textnormal{\hyperref[int]{\textbf{(E)}}} and \textnormal{\hyperref[reg]{\textbf{(R)}}}, and that $(\Omega,\Fcal,\Pbb)$ is a probability space on which is defined an IID sequence $\Sigma_1,\Sigma_2,\ldots$ of $\CovSpace$-valued random variables with common distribution $P$.
These conditions together imply that the optimization problem \eqref{eqn:pop-bary} has a unique solution and that the random optimization problem \eqref{eqn:emp-bary} almost surely has a unique solution; we denote the corresponding solutions by $\bary$ and $\empbary$, and refer to them as population and empirical (Bures-Wasserstein) barycenters, respectively.

These assumptions are mild and they are indeed satisfied in many examples of interest; for the sake of concreteness, we explain some in detail now.

\begin{example}[commuting model]
    Suppose $\mathcal{C}\subseteq\CovSpacePos$ is a set of operators which all commute with each other, and suppose that $P(\mathcal{C})=1$.
    Then, all properties of $\Sigma$ can be reduced to properties of its eigenvalue sequence $\lambda_1(\Sigma)\ge \lambda_2(\Sigma)\ge \cdots$, since all elements $\Sigma\in\mathcal{C}$ can be simultaneously diagonalized, by which $\BW^2(\Sigma,\Sigma') = \sum_{i=1}^{\infty}|\sqrt{\lambda_i(\Sigma)}-\sqrt{\lambda_i(\Sigma')}|^2$.
    For example, the barycenter $M^{\ast}$ is uniquely described by $\lambda_i(M^{\ast}) = (\int_{\CovSpace}(\lambda_i(\Sigma))^{1/2}\diff P(\Sigma))^2$ for all $i$.
    In this case, condition~\textnormal{\hyperref[reg]{\textbf{(R)}}} holds whenever we have $\int_{\CovSpace}\sup_i\lambda_i(\Sigma)/\lambda_i(M^{\ast})\diff P(\Sigma)<\infty$, or equivalently when $\int_{\CovSpace}\|t_{M^{\ast}}^{\Sigma}\|_{\infty}\diff P(\Sigma)<\infty$.
    Condition~\textnormal{\hyperref[int]{\textbf{(E)}}}, as in the general case, concerns the exponential integrability of $\sqrt{\trace(\Sigma)}=(\sum_{i=1}^{\infty}\lambda_i(\Sigma))^{1/2}$.
\end{example}

\begin{example}[template-deformation model]\label{ex:template}
    Fix $M^{\ast}\in\CovSpacePos$ and let $T:\Hilbert\to\Hilbert$ be a random (strictly) positive symmetric linear operator with $\Ebb[T] = \Id$, and define $P$ to be the law of the random element $\Sigma:= TM^{\ast}T$.
    By \cite[Theorem~13]{Masarotto} the population Bures-Wasserstein barycenter of $P$ is $M^{\ast}$, and condition~\textnormal{\hyperref[int]{\textbf{(E)}}} is satisfied whenever $\BW(M^{\ast},\Sigma) = \|T-\Id\|_{M^{\ast}}$ has all exponential moments finite.
    If $\dim(\Hilbert)<\infty$, then condition~\textnormal{\hyperref[reg]{\textbf{(R)}}} is trivially satisfied, but if 
    $\dim(\Hilbert)=\infty$ then things are more delicate; nonetheless, it can be shown that a sufficient condition is $T=\Id+S$ with $\|(M^{\ast})^{-1/2}S(M^{\ast})^{1/2}\|<1$ almost surely.
    This model is, in some sense, the canonical model for the Bures-Wasserstein geometry (see the discussion in \cite{Masarotto}) and it has been useful in statistics applications, e.g., \cite{panaretos2010second,Pigoli}.
\end{example}

\begin{example}[Wishart distribution]
    When $\dim(\Hilbert)=:d<\infty$, the most common model for non-negative symmetric random matrices is the Wishart distribution, defined for scale matrix $V\in\CovSpace$ and degrees of freedom $m\in\Nbb$ as having a density proportional to $(\det(\Sigma))^{(m-d-1)/2}\exp(-\trace(V^{-1}\Sigma)/2)$.
    If $V\in\CovSpacePos$ and $m\ge d$, then condition~\textnormal{\hyperref[int]{\textbf{(E)}}} is trivially satisfied and~\textnormal{\hyperref[reg]{\textbf{(R)}}} is satisfied by~\eqref{eqn:finite-dim-regularity}.
\end{example}

\begin{example}[rotationally-invariant models]\label{ex:invar}
    Suppose $\dim(\Hilbert)=:d<\infty$ and that $P$ is the law of a random element $\Sigma\in\CovSpacePos$ such that $\Sigma$ has the same distribution as $U\Sigma U^{\top}$ for all orthogonal matrices $U\in\Rbb^{d\times d}$.
    Then by symmetry considerations, one can show that the population barycenters must be of the form $\lambda \Id$ for some $\lambda>0$.
    In this case, condition~\textnormal{\hyperref[int]{\textbf{(E)}}} is satisfied whenever $\sqrt{\trace(\Sigma)}$ has all exponential moments finite, and condition~\textnormal{\hyperref[reg]{\textbf{(R)}}} is satisfied by~\eqref{eqn:finite-dim-regularity}. 
    Note that the Wishart distribution is an example of a rotationally-invariant model when the scale matrix $V$ is equal to the identity matrix.
\end{example}

\subsection{Exponential Tilting}\label{subsec:exp-tilt}

In this subsection we develop a canonical notion of exponential tilting for probability measures on $(\CovSpace,\BW)$.
The key is the following fundamental duality result, which shows that relative entropy projections under barycenter constraints are solved by certain exponential tilts.

\begin{proposition}\label{prop:tilting}
    For any $M\succ 0$ the optimization problems
          \begin{equation}\label{eqn:IP}
            \begin{cases}
                \textnormal{maximize} & \trace(AM) - \log\int_{\CovSpace}\exp \trace\left(AMt_{M}^{\Sigma}\right)\diff P(\Sigma) \\
                \textnormal{over} & A\in\SymSpace_M
            \end{cases}
        \end{equation}
	and
 \begin{equation}\label{eqn:JP}
		\begin{cases}
			\textnormal{minimize} &H(Q\, | \, P) \\
			\textnormal{over } & Q\in \Pcal_1(\CovSpacePos) \\
			\textnormal{with} & F(Q) = M
		\end{cases}
    \end{equation}
	have the same value and admit at most one optimizer. Furthermore, feasible points $A\in \SymSpace_M$ and $Q\in\Pcal_1(\CovSpacePos)$ are optimal (for \eqref{eqn:IP} and \eqref{eqn:JP}, respectively) if and only if they satisfy $Q=P^{M\to A}$, where 
         \begin{equation}\label{eqn:exp-tilt-def}
        \frac{\diff P^{M\to A}}{\diff P}(\Sigma) \propto \exp \trace(AMt_{M}^{\Sigma}).
        \end{equation}    
\end{proposition}
It will be useful in what follows to give a name to the solution operators for \eqref{eqn:IP} and \eqref{eqn:JP}.
To do this, we first write
\begin{equation}\label{eqn:effdom}
    \effdom := \{M\in\CovSpacePos: \textnormal{\eqref{eqn:IP} and \eqref{eqn:JP} have solutions}\}
\end{equation}
and then we define $P^{M}$ and $A_M$, for $M\in\effdom$, to be the uniquely-defined solutions to \eqref{eqn:JP} and \eqref{eqn:IP}, respectively, so that clearly $P^M = P^{M\to A_M}$ for all $M\in\effdom$.

The proof of this result requires the following estimates:

\begin{lemma}
	\label{lemma:bounds}
	For any $M,\Sigma\in\CovSpacePos$, we have
        \begin{equation}\label{eqn:tracebound}
		\trace(Mt_{M}^{\Sigma}) = \trace\left((M^{\sfrac{1}{2}}\Sigma M^{\sfrac{1}{2}})^{1 / 2}\right) \le \BW(M,0)\BW(\Sigma,0)
	\end{equation}
        and for $A\in\SymSpace_M$ we have
        \begin{equation}\label{eqn:integrand-bound}
            |\trace(AM(t_{M}^{\Sigma}-\Id))| \le\|A\|_2\BW(M,0)\BW(M,\Sigma).
        \end{equation}
\end{lemma}

\begin{proof}
    Throughout, recall from \eqref{eqn:BWdef} that $\BW^2(M,0) = \trace(M)$ and $\BW^2(\Sigma,0) = \trace(\Sigma)$, and that $\trace(B) = \sum_{i\in\Nbb}\langle Be_i,e_i\rangle = \sum_{i\in\Nbb}\|B^{\sfrac{1}{2}}e_i\|^2$ for any $B\succeq 0$ and any CONS $\{e_i\}_{i\in\Nbb}$ for $\Hilbert$.

    To prove equation~\eqref{eqn:tracebound}, write $K:=(M^{\sfrac{1}{2}}\Sigma M^{\sfrac{1}{2}})^{\sfrac{1}{2}}$.
    For the equality, let $\{e_i\}_{i\in\Nbb}$ be a CONS of eigenvectors of $M$ with eigenvalues $m_i>0$ (this exists since $M$ is compact and $M\succ 0$); since $M^{-\sfrac{1}{2}}e_i = m_i^{-\sfrac{1}{2}}e_i$, we get from \eqref{eqn:optmap} that
    \begin{equation*}
        \langle Mt_{M}^{\Sigma}e_i,e_i\rangle = m_i\langle K M^{-\sfrac{1}{2}}e_i, M^{-\sfrac{1}{2}}e_i\rangle = \langle Ke_i,e_i\rangle
    \end{equation*}
    for all $i\in\Nbb$, and summing over $i\in\Nbb$ gives $\trace(Mt_{M}^{\Sigma}) = \trace(K)$.
    For the inequality, set $X:=\Sigma^{\sfrac{1}{2}}M^{\sfrac{1}{2}}$, so that $X^{\ast}X = K^2$.
    Since $M,\Sigma\succ 0$, both $X$ and $X^{\top} = M^{\sfrac{1}{2}}\Sigma^{\sfrac{1}{2}}$ are injective; in particular, $K$ is injective, since $K^2h = 0$ implies $\|Xh\|^2 = \langle K^2h,h\rangle = 0$.
    Thus, as $K$ is also compact, self-adjoint, and non-negative, there is a CONS $\{f_i\}_{i\in\Nbb}$ of eigenvectors of $K$ with eigenvalues $s_i>0$, and $\trace(K) = \sum_{i\in\Nbb}s_i$.
    Observe that $\langle Xf_i,Xf_j\rangle = \langle K^2f_i,f_j\rangle = s_i^2\delta_{ij}$, so that $g_i:=s_i^{-1}Xf_i$ defines an orthonormal family $\{g_i\}_{i\in\Nbb}$ satisfying $s_i = \langle Xf_i,g_i\rangle$; moreover, it is a CONS, since its span is the range of $X$, which is dense because $X^{\top}$ is injective.
    Hence, applying Cauchy-Schwarz first in $\Hilbert$ and then in $\ell^2(\Nbb)$, we get
    \begin{align*}
        \trace(K) = \sum_{i\in\Nbb}\langle M^{\sfrac{1}{2}}f_i,\Sigma^{\sfrac{1}{2}}g_i\rangle
        &\le \sum_{i\in\Nbb}\|M^{\sfrac{1}{2}}f_i\|\,\|\Sigma^{\sfrac{1}{2}}g_i\| \\
        &\le \left(\sum_{i\in\Nbb}\|M^{\sfrac{1}{2}}f_i\|^2\right)^{\sfrac{1}{2}}\left(\sum_{i\in\Nbb}\|\Sigma^{\sfrac{1}{2}}g_i\|^2\right)^{\sfrac{1}{2}} \\
        &= \trace(M)^{\sfrac{1}{2}}\trace(\Sigma)^{\sfrac{1}{2}} = \BW(M,0)\BW(\Sigma,0)
    \end{align*}
    as claimed.

    To prove equation~\eqref{eqn:integrand-bound}, we may assume $\|A\|_2<\infty$.
    Note that $\trace(AM(t_{M}^{\Sigma}-\Id)) = \langle A,t_{M}^{\Sigma}-\Id\rangle_M$, so Cauchy-Schwarz in $(\SymSpace_M,\langle\,\cdot\,,\,\cdot\,\rangle_M)$ and \eqref{eqn:BWequivMaps} give
    \begin{equation*}
        |\trace(AM(t_{M}^{\Sigma}-\Id))| \le \|A\|_M\|t_{M}^{\Sigma}-\Id\|_M = \|A\|_M\BW(M,\Sigma).
    \end{equation*}
    Moreover, with $\{e_i\}_{i\in\Nbb}$ any CONS, we have $\|A\|_M^2 = \|M^{\sfrac{1}{2}}A\|_2^2 = \|AM^{\sfrac{1}{2}}\|_2^2 = \sum_{i\in\Nbb}\|AM^{\sfrac{1}{2}}e_i\|^2$, where the middle equality holds since $\sum_{i}\|Be_i\|^2 = \sum_{i,j}|\langle Be_i,e_j\rangle|^2 = \sum_j\|B^{\ast}e_j\|^2$ for any operator $B$.
    Also, for any unit vector $h\in\Hilbert$, extending $h$ to a CONS $\{h_j\}_{j\in\Nbb}$ shows $\|Ah\|^2\le\sum_{j\in\Nbb}\|Ah_j\|^2 = \|A\|_2^2$.
    Therefore
    \begin{equation*}
        \|A\|_M^2 \le \|A\|_2^2\sum_{i\in\Nbb}\|M^{\sfrac{1}{2}}e_i\|^2 = \|A\|_2^2\trace(M) = \|A\|_2^2\BW^2(M,0),
    \end{equation*}
    and combining the two displays proves \eqref{eqn:integrand-bound}.
\end{proof}

Now we move on to the proof.

\begin{proof}[Proof of Proposition~\ref{prop:tilting}]
    Consider the injective embedding $\ell_{M}:\CovSpacePos\to \SymSpace_M$  defined in \eqref{eqn:log}.
    We can reparameterize \eqref{eqn:JP} in the following ways:
    first, we use the fixed-point equation \eqref{eqn:pop-fixed-pt} to see that, for every $Q\in \Pcal_2(\CovSpacePos)$, the constraint $F(Q) = M$ is equivalent to $\int_{\SymSpace_M}A\diff (({\ell_M})_{\#}Q)(A) = 0$; second we use \eqref{eqn:BWequivMaps} to see that the integrability constraint $({\ell_M})_{\#}Q\in \Pcal_2(\SymSpace_M)$ is equivalent to $Q\in\Pcal_2(\CovSpacePos)$;
	third, we have $H(({\ell_M})_{\#}Q\, | \, ({\ell_M})_{\#}P)=H(Q\, | \, P)$ for all $P,Q\in\Pcal(\CovSpacePos)$, again using the fact that $\ell_M$ is an injection.
	Thus, we can change variables to $\tilde{Q}:=({\ell_M})_{\#}Q\in \Pcal_2(\SymSpace_M)$, and we have shown that the value of \eqref{eqn:JP} is equal to the value of
	\begin{equation}\label{eqn:thm-LDP-3}
		\begin{cases}
			\textnormal{minimize} &H\left(\tilde{Q}\, \Big| \, ({\ell_M})_{\#}P\right) \\
			\textnormal{over } & \tilde{Q}\in \Pcal_2(\SymSpace_M) \\
			\textnormal{with} & \int_{\SymSpace_M}A\diff \tilde{Q}(A) = 0.
		\end{cases}
	\end{equation}
        Now observe that \eqref{eqn:thm-LDP-3} is a standard optimization problem: over all probability measures on a Hilbert space, minimize the relative entropy from a fixed probability measure subject to a constraint on the expectation.
        It is well-known that this problem is dual to the problem
        \begin{equation}\label{eqn:Hilbert-mean-dual}
		\begin{cases}
			\textnormal{maximize} &\langle A,0\rangle_{M} - \log \int_{\SymSpace_M}\exp\langle A,B\rangle_{M}\diff ((\ell_M)_{\#}P)(B) \\
			\textnormal{over } & A\in\SymSpace_M
		\end{cases}
	\end{equation}
        in the sense that both problems have the same value, both have at most one optimizer, and \eqref{eqn:thm-LDP-3} admits a minimizer if and only if \eqref{eqn:Hilbert-mean-dual} admits a maximizer.
        (See \cite[Problem~12.2]{CoverThomas} for a proof in the discrete setting; in the Hilbert space setting, one can prove this easily from the Donsker-Varadhan variational formula for relative entropy \cite[Lemma~6.2.13]{dembo2009large}.)
        Moreover, we can compute:
        \begin{align*}
            \langle A,0\rangle_{M} - \log \int_{\SymSpace_M}&\exp\langle A,B\rangle_{M}\diff ((\ell_M)_{\#}P)(B) \\
            &= - \log \int_{\SymSpace_M}\exp\langle A,B\rangle_{M}\diff ((\ell_M)_{\#}P)(B) \\
            &= - \log \int_{\CovSpace}\exp\langle A,t_{M}^{\Sigma}-\Id\rangle_{M}\diff P(\Sigma) \\
            &= \trace(AM) - \log \int_{\CovSpace}\exp \trace(AMt_M^\Sigma)\diff P(\Sigma).
        \end{align*}
        Therefore, the value of \eqref{eqn:Hilbert-mean-dual} is equal to the value $I_P(M)$ of \eqref{eqn:IP}.
        This proves the first two statements.

    It remains to prove the characterization of optimality provided in the last statement.
    First, suppose that feasible points $A\in\SymSpace_M$ and $Q\in\Pcal_2(\CovSpacePos)$ satisfy \eqref{eqn:exp-tilt-def}.
    Then, writing
    \begin{equation*}
        Z:= \log\int_{\CovSpace}\exp \trace\left(AMt_{M}^{\Sigma}\right)\diff P(\Sigma)
    \end{equation*}
    for the normalizing constant implicit in \eqref{eqn:exp-tilt-def}, we see that $Q$ is defined via
    \begin{equation}
        \frac{\diff Q}{\diff P}(\Sigma) = \exp (\trace(AMt_{M}^{\Sigma}) - Z),
    \end{equation}
    hence we can compute
    \begin{equation}\label{eqn:rel-int-calc}
        \begin{split}
            H(Q|P) &= \int_{\CovSpace} \log \left(\frac{\diff Q}{\diff P}(\Sigma)\right)\diff Q(\Sigma) \\
            &= \int_{\CovSpace} (\trace(AMt_{M}^{\Sigma})-Z)\diff Q(\Sigma) \\
            &= \int_{\CovSpace} \trace(AMt_{M}^{\Sigma})\diff Q(\Sigma) - Z \\
            &= \trace\left(AM\int _{\CovSpace}t_{M}^{\Sigma}\diff Q(\Sigma)\right) - Z \\
            &= \trace\left(AM\int _{\CovSpace}t_{M}^{\Sigma}\diff Q(\Sigma)\right) - \log\int_{\CovSpace}\exp \trace\left(AMt_{M}^{\Sigma}\right)\diff P(\Sigma) \\
            &= \trace\left(AM\right) - \log\int_{\CovSpace}\exp \trace\left(AMt_{M}^{\Sigma}\right)\diff P(\Sigma).
        \end{split}
    \end{equation}
    This shows that the value $H(Q|P)$ is achievable by a feasible point $A\in\SymSpace_{M}$ for \eqref{eqn:IP}, hence $A$ and $Q$ are optimal for their respective problems.
    
    For the opposite direction, suppose that $A\in\SymSpace_M$ and $Q\in\Pcal_2(\CovSpacePos)$ are optimal.
    Then, we characterize $Q$ by developing first-order optimality conditions for \eqref{eqn:IP}.
    Of course, for fixed $M,\Sigma\in\CovSpacePos$, the function
    \begin{equation*}
        f_{M,\Sigma}(A):=\exp \trace\left(AM(t_{M}^{\Sigma}-\Id)\right)
    \end{equation*}
    is Fr\'echet differentiable in $A\in\SymSpace_M$, and its \Frechet derivative is given by
    \begin{equation*}
        \nabla_A f_{M,\Sigma}(A)(H) = \trace(HM(t_{M}^{\Sigma}-\Id))f_{M,\Sigma}(A),
    \end{equation*}
    for $H\in\SymSpace_M$.
    Towards differentiating under the integral, note that we have
    \begin{align*}
        &\left|\frac{f_{M,\Sigma}(A+\delta H)-f_{M,\Sigma}(A)}{\delta}-\trace(HM(t_{M}^{\Sigma}-\Id))f_{M,\Sigma}(A)\right| \\
        &=\exp\trace\left(AM(t_{M}^{\Sigma}-\Id)\right)\left|\frac{\exp(\delta\trace(HM(t_{M}^{\Sigma}-\Id)))-1-\delta\trace(HM(t_{M}^{\Sigma}-\Id))}{\delta}\right| \\
        &\le\exp\trace\left(AM(t_{M}^{\Sigma}-\Id)\right)\frac{\delta}{2}|\trace(HM(t_{M}^{\Sigma}-\Id))|^2\exp(\delta |\trace(HM(t_{M}^{\Sigma}-\Id))|),
    \end{align*}
    using the (usual) Taylor series remainder bound $|e^{t\delta}-1-t\delta|\le \frac{1}{2}t^2\delta^2\exp(|t|\delta)$ for all $t,\delta\in\Rbb$.
    Now for $0 < \delta < 1$ we use the generous bound $t^2 \le e^{2|t|}$ along with Lemma~\ref{lemma:bounds} to further this as
    \begin{align*}
        \sup_{0<\delta<1}&\left|\frac{f_{M,\Sigma}(A+\delta H)-f_{M,\Sigma}(A)}{\delta}-\trace(HM(t_{M}^{\Sigma}-\Id))f_{M,\Sigma}(A)\right| \\
        &\le\frac{1}{2}\exp\left(|\trace\left(AM(t_{M}^{\Sigma}-\Id)\right)|+3| \trace(HM(t_{M}^{\Sigma}-\Id))|\right) \\
        &\le\frac{1}{2}\exp\left((\|A\|_2+3\|H\|_2)\BW(M,0)\BW(M,\Sigma)\right).
    \end{align*}
    By the integrability condition \textnormal{\hyperref[int]{\textbf{(E)}}}, the right side is integrable with respect to $P$ when $\Sigma$ ranges over $\CovSpacePos$.
    Therefore, we can apply dominated convergence to differentiate under the integral, yielding:
    \begin{align*}
        &\left(\nabla_{A}\log\int_{\CovSpace}\exp \trace\left(AM(t_{M}^{\Sigma}-\Id)\right)\diff P(\Sigma)\right)(H) \\
        &\qquad= \frac{\int_{\CovSpace}\trace(HM(t_{M}^{\Sigma}-\Id))\exp \trace\left(AM(t_{M}^{\Sigma}-\Id)\right)\diff P(\Sigma)}{\int_{\CovSpace}\exp \trace\left(AM(t_{M}^{\Sigma}-\Id)\right)\diff P(\Sigma)} \\
        &\qquad= \trace\left(HM\int_{\CovSpace}(t_{M}^{\Sigma}-\Id)\diff P^{M\to A}(\Sigma)\right) \\
        &\qquad= \left\langle H, \int_{\CovSpace}(t_{M}^{\Sigma}-\Id)\diff P^{M\to A}(\Sigma)\right\rangle_{M}
    \end{align*}
    for all $H\in\SymSpace_M$, where the integral is a Bochner integral in $\SymSpace_M$.
    Since $M$ is positive and $A$ is optimal, this implies
    \begin{equation}\label{eqn:fixed-pt}
        \int_{\CovSpace}(t_{M}^{\Sigma}-\Id)\diff P^{M\to A}(\Sigma) = 0.
    \end{equation}
    By \eqref{eqn:pop-fixed-pt}, this fixed-point equation implies that $M$ is the Bures-Wasserstein barycenter of $P^{M\to A}$ (that is, $F(P^{M\to A}) = M$), so that $P^{M\to A}$ is feasible for \eqref{eqn:JP}. Further observe, as in \eqref{eqn:rel-int-calc}, that we have $H(P^{M\to A}\,|\,P) = I_P(M)$.
    Since \eqref{eqn:JP} admits at most one minimizer, we conclude that $Q = P^{M\to A}$.
\end{proof}

We now give some comments about this result.
First, from the point of view of optimization, we note that it identifies a duality between the optimization problems \eqref{eqn:IP} and \eqref{eqn:JP}.
While problem~\eqref{eqn:IP} is clearly convex, the proof reveals that problem~\eqref{eqn:JP} is also convex.
Thus, the result of Proposition~\ref{prop:tilting} can be seen as a standard form of strong duality between infinite-dimensional convex optimization problems.
This is also the starting point for some numerical methods related to large deviations theory, which we explore in Subsection~\ref{subsec:numerical}.

Second, we observe that for each $M\in \CovSpacePos$, the family of probability measures $\{P^{M\to A}\}_{A\in\SymSpace_M}$ is analogous to the usual Euclidean notion of an exponentially-tilted family in the following senses:
For one, the result states that, under a constraint that a probability measure has a fixed barycenter, the smallest possible relative entropy from $P$ is achieved within the family $\{P^{M\to A}\}_{A\in\SymSpace_M}$.
For another, it states that $A\in\SymSpace_M$ is optimal for \eqref{eqn:IP} if and only if $P^{M\to A}$ has barycenter $M$.
See Figure~\ref{fig:exp-tilt} for a geometric interpretation.

Finally, note that $I_P$ appearing in \eqref{eqn:IP} can, in some sense, be expressed as a Fenchel-Legendre transform.
Indeed, notice that $I_P(M) = (\Lambda_P^{M})^{\ast}(0)$, where $\Lambda_P^M:(\SymSpace_M,\langle\,\cdot\,,\,\cdot\,\rangle_{M})\to \Rbb$ is
\begin{equation*}
	\Lambda_P^{M}(A):=\log\int_{\CovSpace}\exp \langle A,t_{M}^{\Sigma}-\Id\rangle_{M}\diff P(\Sigma)
\end{equation*}
for $A\in\SymSpace_M$.
Of course, $\Lambda_P^{M}$ is nothing more than the cumulant generating function for the pushforward of $P$ by the embedding $\ell_M$ defined in \eqref{eqn:log}.
This further explains the geometric interpretation of exponential tilting shown in Figure~\ref{fig:exp-tilt}, and is closely related to the notion of exponential tilting used in the saddle-point approximation of general $M$-estimators \cite{VecchiaRonchettiIlievski}.

\begin{figure}
	\begin{center}
		\begin{tikzpicture}[scale=1.7]
			
			\coordinate (A) at (-0.2, 0.7);
			\coordinate (B) at (-1.85, 0.3);
			\coordinate (C) at (-2.7, 0.95);
			\coordinate (D) at (-0.9, 1.4);
			\fill[thin, opacity=0.75, color=gray!30] (A) to[out=170,in=40] (B) to[out=120,in=350] (C) to[out=50,in=170] (D) to[out=330,in=120] (A);
			\draw[thin, solid] (A) to[out=170,in=40] (B) to[out=120,in=350] (C) to[out=50,in=170] (D);
			\draw[thin, dashed] (D) to[out=330, in=120] (A);
                \node at (-0.5, 0.5) {$\mathbb{K}$};
			
			\coordinate (D1) at (-1.8, 1.1);
			\coordinate (D2) at (-1.2, 1.25);
			\coordinate (D3) at (-1.85, 0.5);
			\coordinate (D4) at (-1.7, 1.3);
			\coordinate (D5) at (-1.75, 0.7);
			\coordinate (D6) at (-1.9, 0.72);
			\coordinate (D7) at (-1.55, 0.9);
			
			\filldraw[blue!50] (D1) circle (0.02) node[] {};
			\filldraw[blue!50] (D2) circle (0.02) node[] {};
			\filldraw[blue!50] (D3) circle (0.02) node[] {};
			\filldraw[blue!50] (D4) circle (0.02) node[] {};
			\filldraw[blue!50] (D5) circle (0.02) node[] {};
			\filldraw[blue!50] (D6) circle (0.02) node[] {};
			\filldraw[blue!50] (D7) circle (0.02) node[] {};
			\node[color=blue] at (-2.0, 1.65) {$P$};
			
			\filldraw[red] (-1.6, 1.2) circle (0.02) node[above right] {$M$};

		\end{tikzpicture} \begin{tikzpicture}
			\draw[dashed, -latex] (0,0) to (2,0);
			\node[] at (1,0.5) {};
			\node[] at (0,-1) {};
			\node[] at (0,-1) {};
		\end{tikzpicture}\quad\begin{tikzpicture}[scale=1.7]
			
			\coordinate (A) at (-0.2, 0.7);
			\coordinate (B) at (-1.85, 0.3);
			\coordinate (C) at (-2.7, 0.95);
			\coordinate (D) at (-0.9, 1.4);
			\fill[thin, opacity=0.75, color=gray!30] (A) to[out=170,in=40] (B) to[out=120,in=350] (C) to[out=50,in=170] (D) to[out=330,in=120] (A);
			\draw[thin, solid] (A) to[out=170,in=40] (B) to[out=120,in=350] (C) to[out=50,in=170] (D);
			\draw[thin, dashed] (D) to[out=330, in=120] (A);
			\node at (-0.5, 0.5) {$\mathbb{K}$};
			
			\coordinate (D1) at (-1.8, 1.1);
			\coordinate (D2) at (-1.2, 1.25);
			\coordinate (D3) at (-1.85, 0.5);
			\coordinate (D4) at (-1.7, 1.3);
			\coordinate (D5) at (-1.75, 0.7);
			\coordinate (D6) at (-1.9, 0.72);
			\coordinate (D7) at (-1.55, 0.9);
			
			\filldraw[blue!50, opacity=0.5] (D1) circle (0.02) node[] {};
			\filldraw[blue!75, opacity=0.8] (D2) circle (0.02) node[] {};
			\filldraw[blue!50, opacity=0.2] (D3) circle (0.02) node[] {};
			\filldraw[blue!60, opacity=0.6] (D4) circle (0.02) node[] {};
			\filldraw[blue!50, opacity=0.5] (D5) circle (0.02) node[] {};
			\filldraw[blue!50, opacity=0.5] (D6) circle (0.02) node[] {};
			\filldraw[blue!50, opacity=0.8] (D7) circle (0.02) node[] {};
			\node[color=blue] at (-2.0, 1.65) {};
			\filldraw[red] (-1.6, 1.2) circle (0.02) node[below right] {};
                \node[color=blue] at (-2.0, 1.65) {$P^{M}$};
		\end{tikzpicture} \\
		\begin{tikzpicture}
			\draw[-latex] (0,0) to (0,-2);
			\node[] at (-1, 0) {};
			\node[] at (0.75, -1.) {$\ell_M$};
			
			\draw[latex-] (7,0) to (7,-2);
			\node[] at (7.75, -1.) {$e_M$};
		\end{tikzpicture} \\
	\begin{tikzpicture}[scale=1.7]
			
			\coordinate (A) at (-0.2, 0.7);
			\coordinate (B) at (-1.85, 0.3);
			\coordinate (C) at (-2.7, 0.95);
			\coordinate (D) at (-0.9, 1.4);
			\fill[thin, opacity=0.75, color=gray!30] (A) to[out=170,in=40] (B) to[out=120,in=350] (C) to[out=50,in=170] (D) to[out=330,in=120] (A);
			\draw[thin, solid] (A) to[out=170,in=40] (B) to[out=120,in=350] (C) to[out=50,in=170] (D);
			\draw[thin, dashed] (D) to[out=330, in=120] (A);
			\node at (-0.5, 0.5) {$\mathbb{K}$};
                \node[color=red] at (-2.0, 1.65) {$\SymSpace_{M}$};
			
			\coordinate (D1) at (-1.8, 1.1);
			\coordinate (D2) at (-1.2, 1.25);
			\coordinate (D3) at (-1.85, 0.5);
			\coordinate (D4) at (-1.7, 1.3);
			\coordinate (D5) at (-1.75, 0.7);
			\coordinate (D6) at (-1.9, 0.72);
			\coordinate (D7) at (-1.55, 0.9);

			\filldraw[thin, opacity=0.75, color=red!30, xscale=-0.5, yscale=0.25,shift={(1.75,4.25)}] (0, 0) to (2, -1) to (3,1) to (1, 2) to cycle;
			\draw[thin, color=red, xscale=-0.5, yscale=0.25,shift={(1.75,4.25)}] (0, 0) to (2, -1) to (3,1) to (1, 2) to cycle;
			\filldraw[red] (-1.6, 1.2) circle (0.02) node[below right] {};
			
			\filldraw[blue!50] (-1.8,1.2) circle (0.02) node[] {};
			\filldraw[blue!50] (-1.21,1.35) circle (0.02) node[] {};
			\filldraw[blue!50] (-1.85,0.9) circle (0.02) node[] {};
			\filldraw[blue!50] (-1.7,1.35) circle (0.02) node[] {};
			\filldraw[blue!50] (-1.75,1.0) circle (0.02) node[] {};
			\filldraw[blue!50] (-1.9,1.0) circle (0.02) node[] {};
			\filldraw[blue!50] (-1.55,1.1) circle (0.02) node[] {};
		\end{tikzpicture} \begin{tikzpicture}
			\draw[-latex] (0,0) to (2,0);
			\node[] at (0,-1) {};
		\end{tikzpicture} \quad\begin{tikzpicture}[scale=1.7]
			
			\coordinate (A) at (-0.2, 0.7);
			\coordinate (B) at (-1.85, 0.3);
			\coordinate (C) at (-2.7, 0.95);
			\coordinate (D) at (-0.9, 1.4);
			\fill[thin, opacity=0.75, color=gray!30] (A) to[out=170,in=40] (B) to[out=120,in=350] (C) to[out=50,in=170] (D) to[out=330,in=120] (A);
			\draw[thin, solid] (A) to[out=170,in=40] (B) to[out=120,in=350] (C) to[out=50,in=170] (D);
			\draw[thin, dashed] (D) to[out=330, in=120] (A);
			\node at (-0.5, 0.5) {$\mathbb{K}$};
			
			\coordinate (D1) at (-1.8, 1.1);
			\coordinate (D2) at (-1.2, 1.25);
			\coordinate (D3) at (-1.85, 0.5);
			\coordinate (D4) at (-1.7, 1.3);
			\coordinate (D5) at (-1.75, 0.7);
			\coordinate (D6) at (-1.9, 0.72);
			\coordinate (D7) at (-1.55, 0.9);
			
			\filldraw[thin, opacity=0.75, color=red!30, xscale=-0.5, yscale=0.25,shift={(1.75,4.25)}] (0, 0) to (2, -1) to (3,1) to (1, 2) to cycle;
			\draw[thin, color=red, xscale=-0.5, yscale=0.25,shift={(1.75,4.25)}] (0, 0) to (2, -1) to (3,1) to (1, 2) to cycle;
			\filldraw[red] (-1.6, 1.2) circle (0.02) node[below right] {};
			\draw[->, color=red] (-1.6,1.2) to (-1.4,1.25);
			\filldraw[red] (-1.5, 1.1) circle (0.0) node[right] {$A_M$};
			
			\filldraw[blue!50, opacity=0.5] (-1.8,1.2) circle (0.02) node[] {};
			\filldraw[blue!75, opacity=0.8] (-1.21,1.35) circle (0.02) node[] {};
			\filldraw[blue!50, opacity=0.2] (-1.85,0.9) circle (0.02) node[] {};
			\filldraw[blue!60, opacity=0.6] (-1.7,1.35) circle (0.02) node[] {};
			\filldraw[blue!50, opacity=0.5] (-1.75,1.0) circle (0.02) node[] {};
			\filldraw[blue!50, opacity=0.5] (-1.9,1.0) circle (0.02) node[] {};
			\filldraw[blue!50, opacity=0.8] (-1.55,1.1) circle (0.02) node[] {};

                \node[color=white] at (-2.0, 1.65) {$P^{M\to A}$};
                \node[color=red] at (-2.0, 1.65) {$\SymSpace_{M}$};
		\end{tikzpicture}
	\end{center}
	\caption{Exponential tilting in the Bures-Wasserstein space $(\CovSpace,\BW)$.
    For each covariance $M\succ 0$, we transform $P$ to $P^{M}$ (top) by re-weighting each point $\Sigma\in\CovSpace$ according to the value of $\exp\trace(A_MMt_{M}^{\Sigma}) \propto \exp\trace(A_MM\ell_M(\Sigma))$, where $A_M\in\SymSpace_M$ is chosen so that $P^M$ has Bures-Wasserstein barycenter $M$ (bottom).}

	\label{fig:exp-tilt}
\end{figure}

\subsection{Large Deviations Principle}\label{subsec:LDP}

Armed with the duality in Proposition~\ref{prop:tilting}, we can now state and prove the main result of the paper, establishing the large deviations principle for Bures-Wasserstein barycenters, and prove some fundamental properties of the rate function.

We define the function $I_P:\CovSpacePos\to [0,\infty]$ via
    \begin{equation}\label{eqn:rate-func-def}
         I_P(M) = \sup_{A\in\SymSpace_M}\left(\trace(AM) - \log\int_{\CovSpace}\exp \trace\left(AMt_{M}^{\Sigma}\right)\diff P(\Sigma)\right),
     \end{equation}
     for $M\in\CovSpacePos$, and we observe that this is exactly the optimal value of equation~\eqref{eqn:IP} of Proposition~\ref{prop:tilting}.
     We also defined the set $\effdom$ in equation~\eqref{eqn:effdom}, and we note that it is contained in the \textit{effective domain} of the function $I_P$, that is $\effdom = \{M\in\CovSpacePos: I_P(M)<\infty\}$.

Let us say that the function $I_P:\CovSpacePos\to [0,\infty]$ is a \textit{good rate function} if its sub-level sets are compact, and that $\{\empbary\}_{n\in\Nbb}$ satisfy \textit{a large deviations principle (LDP) with respect to $I_P:\CovSpacePos\to [0,\infty]$} if for all measurable $E\subseteq \CovSpacePos$ we have
\begin{equation}\label{eqn:LDP}
    \begin{split}
        -\inf\{I_P(M): M\in E^{\circ}\} &\le \liminf_{n\to\infty}\frac{1}{n}\log \Pbb(\empbary\in E) \\
        &\le \limsup_{n\to\infty}\frac{1}{n}\log \Pbb(\empbary\in E) \le -\inf\{I_P(M): M\in \bar E\},
    \end{split}
\end{equation}
where $E^{\circ}$ and $\bar E$ denote the interior and closure of $E$ with respect to $\BW$, respectively.

\begin{theorem}\label{thm:main}
    $\{\empbary\}_{n\in\Nbb}$ satisfy an LDP in $(\CovSpacePos,\BW)$ with good rate function $I_P$.
\end{theorem}

\begin{proof}
    By condition~\textnormal{\hyperref[int]{\textbf{(E)}}} and \cite[Theorem~1]{SanovWasserstein}, the empirical measures $\{\bar P_n\}_{n\in\Nbb}$, defined as
    $$
    \bar P_n := \frac{1}{n}\sum_{i=1}^n\delta_{\Sigma_i},
    $$
    for $n\in\Nbb$, satisfy an LDP in $(\Pcal_1(\CovSpace),W_1)$ with good rate function $H(\,\cdot\, |\, P):\Pcal(\CovSpace)\to[0,\infty]$, where $H$ denotes the relative entropy between probability measures on $(\CovSpace,\BW)$.
    Moreover, by Corollary~\ref{cor:cty-uniqueness}, the map $F:(\Pcal_1(\CovSpacePos),W_1)\to (\CovSpacePos,\BW)$ is continuous.
    Thus, the contraction principle \cite[Theorem~4.2.1]{dembo2009large} implies that $\{F(\bar P_n)\}_{n\in\Nbb} = \{M_n^{\ast}\}_{n\in\Nbb}$ satisfies an LDP in $(\CovSpacePos,\BW)$ with good rate function
    \begin{equation*}
		J_P(M) := \inf\left\{H(Q\,|\,P): Q\in \Pcal_1(\CovSpace), F(Q) = M\right\}.
    \end{equation*}
    We thus conclude by Proposition \ref{prop:tilting}, which shows that $J_P = I_P$.
\end{proof}

\begin{remark}\label{rem:W1-v-W2}
    The proof of Theorem~\ref{thm:main} transforms a suitable LDP for the empirical measures $\{\bar P_n\}_{n\in\Nbb}$ into the desired LDP for the empirical barycenters $\{M_n^{\ast}\}_{n\in\Nbb}$ by applying suitable continuity of the map $F:\Pcal(\CovSpacePos)\to \CovSpacePos$.
    However, there are some subtleties to this argument.

    The exponential integrability condition~\textnormal{\hyperref[int]{\textbf{(E)}}} implies only that the empirical measures satisfy an LDP in $(\Pcal_1(\CovSpace),W_1)$.
    So, in order to apply the contraction principle, we require continuity of the function $F:(\Pcal_1(\CovSpace),W_1)\to(\CovSpacePos,\BW)$.
    However, this does not follow from most known results on Wasserstein barycenters: the celebrated result of Le Gouic--Loubes \cite{le2017existence} provides only $W_2$ continuity and requires the underlying metric space to be (complete, geodesic, and) locally compact; the recent result of \cite[Remark~4.1 and Theorem~4.2]{PanaretosSantoroBW} by the second author allows for $\dim(\Hilbert) =\infty$ but still only provides $W_2$ continuity.
    We require the infinite-dimensional $W_1$ case in our Corollary~\ref{cor:cty-uniqueness}, which is closely related to the work \cite{JaffeInfDim} by the first author.
    
    This $W_1$ continuity is what allows us to achieve the minimal moment condition~\textnormal{\hyperref[int]{\textbf{(E)}}}.
    Indeed, applying the contraction principle to the function $F:(\Pcal_2(\CovSpace),W_2)\to(\CovSpacePos,\BW)$ requires that the empirical measures satisfy an LDP in $(\Pcal_2(\CovSpace),W_2)$ which is equivalent \cite{SanovWasserstein} to the stricter moment condition    \begin{equation}\label{eqn:stronger-moment}
    \int_{\CovSpace}\exp(\lambda \BW^2(\Sigma,0))\,\textnormal{d}P(\Sigma) < \infty \qquad\textnormal{ for all }\qquad \lambda > 0,
    \end{equation}
    meaning all \emph{squared} exponential moments are finite.
    However, condition \eqref{eqn:stronger-moment} is very strict and does not even cover the setting that $\BW(\Sigma,0)$ has sub-Gaussian tails, which will be the focus of our applications.
\end{remark}

We will now observe that $I_P$, in some sense, contains within it the usual rate function from Cram\'er's theorem.
This should not come as a surprise, since $(\CovSpace,\BW)$ contains subsets which are isometric to a Euclidean space; for instance, geodesic segments are isometric to Euclidean intervals.
To see this manifest in the rate function, consider the special case that $P$ is concentrated on a geodesic segment $\gamma:[0,1]\to \CovSpacePos$.
(Notice that this is a special case of the template-deformation model of Example~\ref{ex:template}.)
On the one hand, it is easy to show that we have $I_P(M) = \infty$ for $M\notin \gamma$.
On the other hand, for $M\in\gamma$ we can make some simplifications.
We may express $t_M^{\Sigma}-\Id = (\gamma^{-1}(\Sigma)-\gamma^{-1}(M))\gamma'(0)|\gamma|$ where $|\gamma|$ is the length of $\gamma$, and we note by Proposition~\ref{prop:tilting} that a maximizing $A\in\SymSpace_M$ must be aligned to $\gamma$, hence we can substitute $A = a\gamma'(0)$.
Thus we have:
\begin{align*}
	I_P(M)
	&= \sup_{a\in \Rbb}\left( - \log\int_{\CovSpace}\exp \trace\left(a\gamma'(0)M(\gamma^{-1}(\Sigma) - \gamma^{-1}(M))|\gamma|\gamma'(0)\right)\diff P(\Sigma)\right) \\
	&= \sup_{a\in \Rbb}\left( - \log\int_{\CovSpace}\exp\left( a|\gamma|\,\|\gamma'(0)\|_M^2 (\gamma^{-1}(\Sigma) - \gamma^{-1}(M))\right)\diff P(\Sigma)\right) \\
	&= \sup_{a'\in \Rbb}\left(a'\gamma^{-1}(M) - \log\int_{\CovSpace}\exp\left( a' \gamma^{-1}(\Sigma) \right)\diff P(\Sigma)\right)
\end{align*}
for all $M\in\gamma$, where we applied the change of variables $a' = a|\gamma|\|\gamma'(0)\|_M^2$.
Now we see that $I_P\circ \gamma^{-1}$ is exactly the Fenchel-Legendre transform of the cumulant generating function of $\gamma^{-1}(\Sigma)$ in $\Rbb$.

We conclude this section by proving some of the fundamental properties of the rate function $I_P$.

\begin{proposition}\label{prop:rate-function-properties}
    The function $I_P:\CovSpacePos\to[0,\infty]$ enjoys the following properties:
    \begin{itemize}
        \item[(a)] $I_P(M) = 0$ if and only if $M=\bary$.
        \item[(b)] $I_P$ is lower semi-continuous.
        \item[(c)] $I_P$ is coercive and satisfies $I_P(M)/\BW(M,0) \to \infty$ as $\BW(M,0)\to\infty$.
        \item[(d)] $I_P$ is strictly convex along continuous paths $\gamma:[0,1]\to\CovSpacePos$ of the form
        $$
        \gamma(t) := F\left((1-t)P^{M_0} + tP^{M_1}\right)
        $$
        for $M_0,M_1\in\effdom$.
    \end{itemize}
\end{proposition}

Before we get into the proof of this result, let us give some remarks.
First, we note that (a) and (b) above can already be deduced from the proof of Theorem~\ref{thm:main}, but we find it more instructive to prove them directly from the definition of $I_P$.
 Second, we note that these properties are largely analogous to the properties of the rate function appearing in Cram\'er's theorem.
Finally, note that in part (d) we establish that $I_P$ is convex along certain interpolating paths in $(\CovSpacePos,\BW)$, but we do not show that $I_P$ is geodesically convex; in fact, we conjecture this to be false.
Nonetheless, we believe that the convexity along paths of the form given in part (d) is sufficient for numerical purposes, and we return to this in Subsection \ref{subsec:numerical}.

\begin{proof}[Proof of Proposition~\ref{prop:rate-function-properties}]
    For (a), let's first show that $I_P(\bary) = 0$.
    To do this, we simply use Jensen's inequality and the fixed-point equation \eqref{eqn:pop-fixed-pt} to get:
    \begin{align*}
        I_P(\bary) &= \sup_{A\in\SymSpace_{\bary}} - \log\int_{\CovSpace}\exp \trace\left(A\bary(t_{\bary}^{\Sigma}-\Id)\right)\diff P(\Sigma) \\
        &\le \sup_{A\in\SymSpace_{\bary}} -\int_{\CovSpace}\trace\left(A\bary(t_{\bary}^{\Sigma}-\Id)\right)\diff P(\Sigma) \\
        &= \sup_{A\in\SymSpace_{\bary}} -\trace\left(A\bary\int_{\CovSpace}(t_{\bary}^{\Sigma}-\Id)\diff P(\Sigma)\right) = 0,
    \end{align*}
    as claimed.
    Second, let's assume $I_P(M) = 0$ for $M\in\CovSpacePos$ and use this to show that $M=\bary$.
    By Proposition~\ref{prop:tilting}, there exists some $A\in\SymSpace_M$ such that
    \begin{equation*}
        -\log\int_{\CovSpace}\exp \trace\left(AM(t_{M}^{\Sigma}-\Id)\right)\diff P(\Sigma) = I_P(M) = 0.
    \end{equation*}
    But $0\in\SymSpace_M$ also achieves this value, and Proposition~\ref{prop:tilting} guarantees that the maximizer is unique.
    Therefore, we have $A=0$, and Proposition~\ref{prop:tilting} implies that $M$ is the Bures-Wasserstein barycenter of $P^{M\to 0} = P$.
    Thus, $M=\bary$.

    We now turn to (b). Let us first consider the finite-dimensional case $\Hilbert=\Rbb^d$; we first show that the function
    \begin{equation*}
        M\mapsto \int_{\CovSpace}\exp \trace\left(AM(t_{M}^{\Sigma}-\Id)\right)\diff P(\Sigma)
    \end{equation*}
    is continuous for each fixed $A\in\SymSpace$.
    To do this, define the function
    \begin{equation*}
        f(A,\Sigma,M):=\exp \trace\left(AM(t_{M}^{\Sigma}-\Id)\right),
    \end{equation*}
    and we note that this is continuous as a function of $M\in\CovSpacePos$, for fixed $A\in\SymSpace$ and $\Sigma\in\CovSpacePos$ by the differentiability result in \cite{Kroshnin}.

    Now suppose that $\{M_n\}_{n\in\Nbb}$ and $M$ in $\CovSpacePos$ have $\BW(M_n, M)\to 0$ as $n\to\infty$.
    Then use equation~\eqref{eqn:integrand-bound} of Lemma~\ref{lemma:bounds} and the triangle inequality to get
    \begin{align*}
        \trace(AM_n(t_{M_n}^{\Sigma}-\Id)) &\le\|A\|_2\BW(M_n,0)\BW(M_n,\Sigma) \\
        &\le\|A\|_2\BW(M_n,0)(\BW(M_n,M) + \BW(M,\Sigma))
    \end{align*}
    for all $\Sigma\in\CovSpacePos$.
    Now $\{M_n\}_{n\in\Nbb}$ converging implies that it is bounded, hence
    \begin{equation*}
        \lambda := \max\left\{\sup_{n\in\Nbb}\BW(M_n,0),\sup_{n\in\Nbb}\BW(M_n,M)\right\} < \infty.
    \end{equation*}
    Therefore, we have the uniform upper bound:
    \begin{equation*}
        \sup_{n\in\Nbb}\exp \trace(AM_n(t_{M_n}^{\Sigma}-\Id)) \le \exp\left(\lambda^2\|A\|_2\right)\exp\left(\lambda\|A\|_2\BW(M,\Sigma)\right).
    \end{equation*}
    Condition~\textnormal{\hyperref[int]{\textbf{(E)}}} guarantees that the right side is $P$-integrable over $\Sigma\in\CovSpacePos$, so dominated convergence applies, and this proves (b) in the finite-dimensional case.

    To prove (b) in general, we use a bit more caution.
    Let $\{e_i\}_{i\in\Nbb}$ be a complete orthonormal system (CONS) in $\Hilbert$, and for $k\in\Nbb$ write $\mathsf{P}_k:=\sum_{i=1}^{k}e_i\otimes e_i:\Hilbert \to \Hilbert$ for the rank-$k$ projection operator.
    Now suppose that $\{M_n\}_{n\in\Nbb}$ and $M$ in $\CovSpacePos$ have $\BW(M_n,M)\to 0$.
    Use Proposition~\ref{prop:tilting} to get some $A\in\SymSpace_M$ achieving the supremum in the definition of $I_P(M)$, and let $A_k:= \mathsf{P}_k A\mathsf{P}_k$ for all $k\in\Nbb$, which we know satisfy $\|A_k-A\|_M\to 0$ as $k\to\infty$ by \cite[][Lemma 5]{Masarotto}.
    First observe that, similarly to the above, the use of dominated convergence entails that:
    \begin{align*}
        \lim_{k\rightarrow \infty}&\left(  \trace(A_kM) - \log\int_{\CovSpace}\exp \trace\left(A_k Mt_{M}^{\Sigma}\right)\diff P(\Sigma)\right) \\
        &=\trace(A M) - \log\int_{\CovSpace}\exp \trace\left(A Mt_{M}^{\Sigma}\right)\diff P(\Sigma) = I_P(M).
    \end{align*}
    Now let $\varepsilon>0$ be arbitrary.
    By the previous display, there exists $k\in\Nbb$ sufficiently large so that we have
    \begin{align*}
        \trace(A_kM) - \log\int_{\CovSpace}\exp \trace\left(A_k Mt_{M}^{\Sigma}\right)\diff P(\Sigma) > I_P(M) - \varepsilon.
    \end{align*}
    We see that $A_k$ has $\|A_k\|_2 < \infty$, so it follows that $A_k\in\SymSpace_{M_n}$ for all $n\in\Nbb$. Furthermore, note that:
    for any $k\in\Nbb$, we have that
    \begin{align*}
        \trace(A_kMt_M^{\Sigma} - A_kM_nt_{M_n}^{\Sigma}) 
    \leq& \|A_k\|_2 \trace( Mt_M^{\Sigma} - M_nt_{M_n}^{\Sigma}) 
    \\=& \|A_k\|_2 \trace( (M^{\sfrac{1}{2}}\Sigma M^{\sfrac{1}{2}})^{\sfrac{1}{2}} - (M_n^{\sfrac{1}{2}}\Sigma M_n^{\sfrac{1}{2}}) ^{\sfrac{1}{2}}) \to 0,
    \end{align*}    
    as $n\to\infty$.
    Therefore, again (similarly) by dominated convergence:    
    \begin{align*}
        \lim_{n\to\infty}&\left(\trace(A_kM_n) - \log\int_{\CovSpace}\exp \trace\left(A_k M_nt_{M_n}^{\Sigma}\right)\diff P(\Sigma)\right) \\
        &= \trace(A_kM) - \log\int_{\CovSpace}\exp \trace\left(A_k Mt_{M}^{\Sigma}\right)\diff P(\Sigma).
    \end{align*}
    In particular, the preceding two displays show:
    \begin{align*}
        \liminf_{n\to\infty}I_P(M_n) &= \liminf_{n\to\infty}\sup_{A\in \SymSpace_{M_n}}\left(\trace(AM_n) - \log\int_{\CovSpace}\exp \trace\left(A M_nt_{M_n}^{\Sigma}\right)\diff P(\Sigma)\right) \\
        &\ge \lim_{n\to\infty}\left(\trace(A_kM_n) - \log\int_{\CovSpace}\exp \trace\left(A_k M_nt_{M_n}^{\Sigma}\right)\diff P(\Sigma)\right) \\
        &= \trace(A_kM) - \log\int_{\CovSpace}\exp \trace\left(A_k Mt_{M}^{\Sigma}\right)\diff P(\Sigma) \\
        &> I_P(M) - \varepsilon.
    \end{align*}
    Now taking $\varepsilon\to 0$ shows
    \begin{equation*}
        \liminf_{n\to\infty}I_P(M_n) \ge I_P(M),
    \end{equation*}
    as needed.
    
    Next we prove the coercivity statement (c).
	To do this, take arbitrary $M\in \CovSpacePos$ and $\lambda \ge 0$, and set $A:= \lambda \Id/\BW(M,0)$.
    Then we can use equation \eqref{eqn:tracebound} of Lemma~\ref{lemma:bounds} to compute:
        \begin{align*}
            I_P(M) &\ge \trace\left(\frac{\lambda \Id}{\BW(M,0)}\cdot M\right) - \log\int \exp\left(\frac{\lambda\trace(M t_{M}^{\Sigma})}{\BW(M,0)}\right)\diff P(\Sigma)
		\\
            &= \lambda\BW(M,0) - \log\int \exp\left(\frac{\lambda\trace(M t_{M}^{\Sigma})}{\BW(M,0)}\right)\diff P(\Sigma)
		\\
            &\ge \lambda\BW(M,0) - \log\int \exp\left(\lambda \BW(\Sigma,0)\right)\diff P(\Sigma).
        \end{align*}
        Since the second term is finite by condition~\textnormal{\hyperref[int]{\textbf{(E)}}}, we can divide by $\BW(M,0)$ and take $\BW(M,0)\to\infty$, and thus we have proven
        \begin{equation*}
            \liminf_{\BW(M,0)\to\infty}\frac{I_P(M)}{\BW(M,0)} \ge \lambda.
        \end{equation*}
    	As $\lambda\ge 0$ was arbitrary, taking $\lambda\to \infty$ yields
        \begin{equation*}
            \lim_{\BW(M,0)\to\infty}\frac{I_P(M)}{\BW(M,0)} =\infty,
        \end{equation*}
        as claimed.

        Finally, we prove the convexity statement (d).
        It follows from Corollary~\ref{cor:cty-uniqueness} that $\gamma$ is a continuous curve connecting $M_0$ to $M_1$.
        And, for all $0\le t \le 1$ we can use Proposition~\ref{prop:tilting} and the well-known convexity of relative entropy to compute:
        \begin{align*}
            I_P(\gamma(t)) &= \inf\{H(Q\,|\,P): Q\in\Pcal_2(\CovSpacePos), F(Q) = \gamma(t)\} \\
            &\le H\left((1-t)P^{M_0} + tP^{M_1}\,|\,P\right) \\
            &\le (1-t)H\left(P^{M_0}\,|\,P\right)+tH\left(P^{M_1}\,|\,P\right) \\
            &= (1-t)I_P(M_0) + tI_P(M_1),
        \end{align*}
        as needed.
        In fact, the strict convexity of the relative entropy implies that equality holds in the second inequality above if and only if $P^{M_0}=P^{M_1}$ whence $M_0=M_1$.
\end{proof}

\subsection{Consequences}
\label{subsec:consequences}
In this subsection, we explore some consequences of Theorem~\ref{thm:main} for the study of concentration of measure for Bures-Wasserstein barycenters.
In fact, we can use the theory of large deviations to both analyze the exponential rate of decay of rare events of interest and to describe \textit{how} such rare events occur.
We focus on the rare event $E_{n,r}:=\{\BW(M_n^{\ast},M^{\ast})\ge r\}$ for $r>0$, though we believe that this strategy may be useful for other rare events of interest.

As in the previous subsections, we assume condition~\textnormal{\hyperref[reg]{\textbf{(R)}}} throughout.
However, we will require a slight strengthening of condition~\textnormal{\hyperref[int]{\textbf{(E)}}} which provides more quantitative control on the tails of $\BW(\Sigma,0)$.
That is,
\begin{itemize}
    \item[\namedlabel{int_SG}{\textbf{(E$'$)}}] For some $\lambda > 0$ we have  $\int_{\CovSpace}\exp(\lambda \BW^2(\Sigma,0))\,\textnormal{d}P(\Sigma) < \infty$.
\end{itemize}
Condition~\textnormal{\hyperref[int_SG]{\textbf{(E$'$)}}} is equivalent to the distribution of $\BW(\Sigma,0)$ being sub-Gaussian, which has been also assumed in \cite{Kroshnin}.
Throughout the rest of this subsection, we write $\sigma^2$ for the sub-Gaussian norm of $\BW(\Sigma,0)$, and $\mu:=\int_{\CovSpace}\BW(\Sigma,0)\diff P(\Sigma)$ for its expectation.
Then we have the following:

\begin{corollary}\label{cor:Hoeffding}
    For any $r\ge0$, we have existence of the limit
    \begin{equation*}
        \Phi_P(r):= -\lim_{n\to\infty}\frac{1}{n}\log\Pbb(\BW(M_n^{\ast},M^{\ast}) \ge r),
    \end{equation*}
    and the function $\Phi_P:[0,\infty)\to[0,\infty]$ satisfies
    \begin{equation}\label{eqn:Hoeffding}
        \liminf_{r\to\infty} \frac{\Phi_P(r)}{r^2} \ge \frac{1}{2\sigma^2}.
    \end{equation}
\end{corollary}

\begin{proof}
    To show that the limit exists, we note by Theorem~\ref{thm:main} that it suffices to show that $\{M\in\CovSpacePos: \BW(M,M^{\ast})\ge r\}$ is equal to the closure of its interior.
    Of course, this follows immediately from the fact that $(\CovSpacePos,\BW)$ is a geodesic metric space.
    Thus, it only remains to show that $\Phi_P$ satisfies the given growth bound.
    To do this, let us take $A:=\sigma^{-2}\Id$ in the supremum appearing in the definition of $I_P$, and then use \eqref{eqn:tracebound} of Lemma~\ref{lemma:bounds} to make the following bound for all $M\in\CovSpacePos$:
    \begin{align*}
        I_P(M) &\ge \trace(AM) - \log\int_{\CovSpace}\exp\trace(AMt_{M}^{\Sigma})\diff P(\Sigma) \\
        &= \frac{1}{\sigma^2}\trace(M) - \log\int_{\CovSpace}\exp\left(\frac{1}{\sigma^2}\trace(Mt_{M}^{\Sigma})\right)\diff P(\Sigma) \\
        &\ge \frac{1}{\sigma^2}\trace(M) - \log\int_{\CovSpace}\exp\left(\frac{1}{\sigma^2}\BW(M,0)\BW(\Sigma,0)\right)\diff P(\Sigma).
    \end{align*}
    Notice that the second term on the right is just the cumulant generating function $\Lambda$ of the real-valued random variable $\BW(\Sigma,0)$, evaluated at the point $\sigma^{-2}\BW(M,0)$.
    By sub-Gaussianity we have $\Lambda(\lambda) \le \mu \lambda + \frac{1}{2}\sigma^2\lambda^2$ for all $\lambda \ge 0$, hence we can further the lower bound as:
    \begin{align*}
        I_P(M) &\ge \frac{1}{\sigma^2}\trace(M) - \log\int_{\CovSpace}\exp\left(\frac{1}{\sigma^2}\BW(M,0)\BW(\Sigma,0)\right)\diff P(\Sigma) \\
        &\ge \frac{1}{\sigma^2}\trace(M) - \left(\frac{\mu}{\sigma^2}\BW(M,0) + \frac{1}{2}\sigma^2\left(\frac{1}{\sigma^2}\BW(M,0)\right)^2\right) \\
        &= \frac{1}{2\sigma^2}\BW^2(M,0) - \frac{\mu}{\sigma^2}\BW(M,0) \\
        &= \frac{1}{2\sigma^2}(\BW(M,0)-\mu)^2 - \frac{\mu^2}{2\sigma^2}
    \end{align*}
    Now suppose that $\BW(M,M^{\ast}) \ge r$ for some $r\ge \BW(M^{\ast},0) + \mu$, and note that the preceding display and the triangle inequality yield
    \begin{equation*}
        \frac{1}{2\sigma^2}(r - \BW(M^{\ast},0)-\mu)^2 - \frac{\mu^2}{2\sigma^2} \le \begin{cases}
            \textnormal{minimize } &I_P(M) \\
            \textnormal{over } &M\in \CovSpacePos \\
            \textnormal{with } &\BW(M,M^{\ast}) \ge r.
        \end{cases}
    \end{equation*}
    Therefore, Theorem~\ref{thm:main} gives
    \begin{align*}
        \Phi_P(r) &= -\lim_{n\to\infty}\frac{1}{n}\log\Pbb(\BW(M_n^{\ast},M^{\ast}) \ge r) \\
        &= \inf\{I_P(M): M\in\CovSpacePos, \BW(M,M^{\ast})\ge r\} \\
        &\ge \frac{1}{2\sigma^2}(r - \BW(M^{\ast},0)-\mu)^2 - \frac{\mu^2}{2\sigma^2}
    \end{align*}
    for all $r\ge \BW(M^{\ast},0) + \mu$.
    (Note that this lower bound is trivial unless $r\ge \BW(M^{\ast},0) + 2\mu$.)
    Finally, we divide by $r^2$ and take the limit as $r\to\infty$ to finish the proof.
\end{proof}

\begin{figure}[t]
    \centering
    \includegraphics[width=0.95\linewidth]{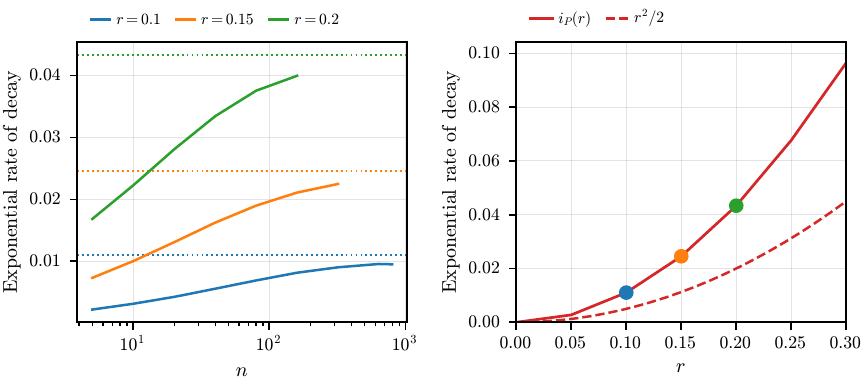}
    \caption{Illustration of Corollary~\ref{cor:Hoeffding} in a numerical simulation.
    For $r\in\{0.1,0.15,0.2\}$ and $n\to\infty$, we show (left) a Monte Carlo approximation of $-\frac{1}{n}\log \Pbb(\BW(M_n^{\ast},M^{\ast})\ge r)$ (solid lines) and we observe that it converges to a numerical approximation of the value $i_P(r):=\min\{I_P(M): M\in\CovSpacePos,\BW(M,M^{\ast})\ge r\}$ (dotted lines).
    For $0\le r\le 0.3$, we plot (right) our numerical approximation of the function $i_P(r)$ (solid line), and we observe that it lies above the function $r\mapsto\frac{1}{2}r^2$ (dashed line).
    Note that the horizontal dotted lines in the left plot correspond precisely to the colored points on the right plot.}
    \label{fig:concentration}
\end{figure}

See Figure~\ref{fig:concentration} for two illustrations of this result in a basic simulation; we describe the exact model for $P$ and our numerical methods for minimizing the rate function in detail in Subsection~\ref{subsec:numerical}.
First, we show a Monte Carlo estimate of $-\frac{1}{n}\log \Pbb(\BW(M_n^{\ast},M^{\ast})\ge r)$ as a function of $n$, and we note that it converges as $n\to\infty$ to the value $i_P(r):=\min\{I_P(M):M\in\CovSpacePos, \BW(M,M^{\ast})\ge r\}$, as predicted by the large deviations principle.
Notice the gap between the Monte Carlo curves and their asymptotic counterparts; this arises due to the fact that simple Monte Carlo estimation of rare events requires exponentially many trials, and yields a poor estimate when $n$ is large and $r$ is not sufficiently small, relative to our fixed choice of $B\approx 10^5$ trials.
Second, we plot the curve $i_{P}(r)$ for a range of values of $r$, and we observe that it lies above the curve $r\mapsto \frac{1}{2\sigma^2}r^2$, where $\sigma^2=1$ is the sub-Gaussian constant in the model $P$.
Interestingly, we note that Corollary~\ref{cor:Hoeffding} is only an asymptotic statement as $r\to\infty$, but the plot suggests that we have $i_{P}(r)\ge\frac{1}{2}r^2$ for all $r\ge 0$.

It is instructive to compare this result with the similar result in \cite[Corollary~2.2]{Kroshnin}. Indeed, both provide evidence of the concentration of measure phenomenon for barycenters in the Bures-Wasserstein space $(\CovSpace,\BW)$. On the one hand, our Corollary~\ref{cor:Hoeffding} is only asymptotic and says nothing about small deviations of $\BW(M_n^{\ast},M^{\ast})$ above its typical value, while \cite[Corollary~2.2]{Kroshnin} is valid in finite samples and is non-trivial for all large deviations scales.
On the other hand, our Corollary~\ref{cor:Hoeffding} is dimension-free and the exponential rate of decay is simply of Hoeffding type, while the rate of decay in \cite[Corollary~2.2]{Kroshnin} is dimension-dependent and much more complicated.

We also emphasize that the rate of decay in \eqref{eqn:Hoeffding} of Corollary~\ref{cor:Hoeffding} cannot, in general, be improved.
This is because any geodesic segment in $(\CovSpace,\BW)$ is isometric to an interval $[0,L]$ under its usual metric for some $L>0$, and it is known in this Euclidean setting that Hoeffding's inequality is tight.
(Relatedly, see the comments after the proof of Theorem~\ref{thm:main}.)

Second, we consider the conditional distribution of $M_n^{\ast}$ given $E_{n,r}$.
Specifically:

\begin{corollary}\label{cor:degenerate}
    If $r>0$, then, for every $\varepsilon>0$, there exists a constant $C_{P,r,\varepsilon}>0$ such that we have
    \begin{equation*}
        \Pbb\left(\BW(M_n^{\ast},M^{\ast})\ge r+ \varepsilon\,\Big|\,\BW(M_n^{\ast},M^{\ast})\ge r\right) \le \exp(-C_{P,r,\varepsilon}n)
    \end{equation*}
    for sufficiently large $n\in\Nbb$.
\end{corollary}

\begin{proof}
    By Theorem~\ref{thm:main} and the contraction principle \cite[Theorem~4.2.1]{dembo2009large}, the sequence $\{\BW(M_n^{\ast},M^{\ast})\}_{n\in\Nbb}$ satisfies a large deviations principle in $[0,\infty)$ with good rate function given by
    \begin{equation*}
        i_P(t) := \inf\{I_P(M): M\in\CovSpacePos, \BW(M,M^{\ast}) = t\}
    \end{equation*}
    for $t\ge 0$.
    Moreover, the intervals $[r,\infty)$ and $[r+\varepsilon,\infty)$ are equal to the closure of their interior, so we have
    \begin{align*}
        \lim_{n\to\infty}&\frac{1}{n}\log\,\Pbb\left(\BW(M_n^{\ast},M^{\ast})\ge r+\varepsilon\,\Big|\, \BW(M_n^{\ast},M^{\ast})\ge r\right) \\
        &= \lim_{n\to\infty}\frac{1}{n}\log\left(\frac{\Pbb(\BW(M_n^{\ast},M^{\ast})\ge r+\varepsilon)}{\Pbb(\BW(M_n^{\ast},M^{\ast})\ge r)}\right) \\
        &= \lim_{n\to\infty}\left(\frac{1}{n}\log\Pbb(\BW(M_n^{\ast},M^{\ast})\ge r+\varepsilon) - \frac{1}{n}\log \Pbb(\BW(M_n^{\ast},M^{\ast})\ge r) \right) \\
        &=-\left(\inf\{i_P(t): t\ge r+\varepsilon\} -\inf\{i_P(t): t\ge r\}\right),
    \end{align*}
    and all of the limits above exist.
    Now write
    \begin{align*}
        i_{r+\varepsilon} &:= \inf\{i_P(t): t \ge r+\varepsilon\} \\
        i_{r} &:= \inf\{i_P(t): t \ge r\},
    \end{align*}
    and observe that it suffices to show $i_{r+\varepsilon} > i_{r}$.
    Since we of course have $i_{r+\varepsilon} \ge i_{r}$, we proceed by assuming $i_{r+\varepsilon} = i_{r}$.
    Then, combining this with the goodness of the rate function $I_P$ shows that there exists $M\in\CovSpacePos$ with $\BW(M,M^{\ast})\ge r+\varepsilon$ and $I_P(M) = i_r$.
    Now consider the path $\gamma:[0,1]\to\CovSpace$ defined via
    \begin{equation*}
        \gamma(t):=F\left((1-t)P^{M} + tP^{M^{\ast}}\right)
    \end{equation*}
    for $0\le t \le 1$.
    By parts (a) and (d) of Proposition~\ref{prop:rate-function-properties}, we know that $I_P(M^{\ast}) = 0$ and that $I_P$ is strictly convex along the path $\gamma$.
    At the same time, we know that $\gamma:[0,1]\to\CovSpace$ is continuous by \cite[Corollary~3.11 and Example~2.15]{JaffeInfDim}.
    Therefore, we see that for sufficiently small $t>0$ we must have both $\BW(\gamma(t),M^{\ast}) > r+\frac{\varepsilon}{2}\ge r$ and $I_P(\gamma(t))\le (1-t)I_P(M) +t\cdot 0 < i_r$.
    This is a contradiction, so we must have $i_{r+\varepsilon} - i_r > 0$ and the result is proved.
\end{proof}

This result is further evidence of concentration of measure, since it shows that the conditional distribution of $M_n^{\ast}$ given $E_{n,r}$ is roughly supported on the boundary of the ball of radius $r$ around $M^{\ast}$.
That is, if $\BW(M_n^{\ast},M^{\ast})\ge r$, then $\BW(M_n^{\ast},M^{\ast})\approx r$ with high probability.
In fact, one can show more generally that the conditional distribution of $M_n^{\ast}$ under a rare event $E$ is approximately supported on $\argmin_{M\in E}I_P(M)$.

\subsection{Numerical Optimization}\label{subsec:numerical}

    We now turn to the question of numerically minimizing the rate function.
    More precisely, we consider numerically solving, either exactly or approximately, the optimization problem
    \begin{equation}\label{eqn:LDP-num}
        \begin{cases}
            \textnormal{minimize} & I_P(M) \\
            \textnormal{over} & M\in \CovSpacePos \\
            \textnormal{with } & M\in E,
        \end{cases}
    \end{equation}
    for a suitable closed set $E\subseteq\CovSpace$.
    In this subsection we discuss a few algorithms for this computation, including the methods used in the computation of Figure~\ref{fig:concentration}. 
    Since in any computational setting we must work in finite dimensions, let us assume throughout this subsection that $\dim(\Hilbert) < \infty$.
    
    In order to use zeroth-order optimization methods for~\eqref{eqn:LDP-num}, we simply evaluate the rate function $I_P(M)$ for a large but finite set of $M$ in $E$.
    Since the maximization over $A$ appearing in \eqref{eqn:rate-func-def} is concave, we see that computing $I_P(M)$ is a convex optimization problem hence it can be solved with the help of standard software packages.
    Thus, to implement this method in practice, we need to be able to choose a fine subset of $E$ on which to evaluate $I_P$, and this must be handled on a case-by-case basis.
    
    In order to use first-order optimization methods for~\eqref{eqn:LDP-num}, it is natural to employ a sort of \emph{projected Riemannian gradient descent (PRGD)} scheme as follows \cite{RiemannianPGD1}:
    Initially, we set $M_0\in E$.
    Then, for a fixed stepsize $\eta>0$, we iterate
    \begin{equation}
      \label{eq:prgd}
      M_{i+1} \;=\; \proj_{\BW}\!\big((\Id-\eta\nabla_{M_i}I_P)M_i(\Id-\eta\nabla_{M_i}I_P);\,E\big),
    \end{equation}
    for $i\in\Nbb$, where $\proj_{\BW}(\,\cdot\,;E):\CovSpacePos\to E$ denotes the metric projection onto $E$ with respect to $\BW$.
    To implement this method in practice, we need to be able to compute both the metric projection and the gradient of the rate function.
    The metric projection must be handled on a case-by-case basis.
    But, provided that $M\mapsto A_M$ is differentiable on $\effdom$, the gradient $\nabla_MI_P$ can be computed explicitly in general. Indeed, note that
    \begin{equation}\label{eqn:rate-func-explicit}
    	I_P(M) = -\log\int_{\CovSpace}\exp \trace(A_M M(t_{M}^{\Sigma}-\Id))\diff P(\Sigma)
    \end{equation}
    by definition.
    Then we use the chain rule to take the gradient of \eqref{eqn:rate-func-explicit} as follows:
    \begin{align*}
    	(\nabla_M I_P)(H) &= -\int_{\CovSpace}\trace((\nabla_M A)(H)M(t_{M}^{\Sigma}-\Id))\diff P^{M}(\Sigma) \\
    	&- \int_{\CovSpace}\trace(A_MH(t_{M}^{\Sigma}-\Id))\diff P^{M}(\Sigma) \\
    	&- \int_{\CovSpace}\trace(A_M M(\nabla_M T^{\Sigma})(H))\diff P^{M}(\Sigma).
    \end{align*}
    Since Proposition~\ref{prop:tilting} implies the fixed-point equation $\int_{\CovSpace}(t_{M}^{\Sigma}-\Id)\diff P^{M}(\Sigma) = 0$, the first two terms above vanish, and it follows that
    \begin{equation}\label{eqn:rate-func-gradient}
    	(\nabla_M I_P)(H) = -\trace\left(A_MM\int_{\CovSpace}\nabla_M T^{\Sigma}(H)\diff P^{M}(\Sigma)\right).
    \end{equation}
    Observe that we needed $\nabla_MA$ to exist in order to execute this argument, but that $\nabla_M A$ does not appear in the final expression for $\nabla_M I_P$.
    Additionally, this expression involves the Jacobian $\nabla_M T^{\Sigma}$, and we remark that this has been computed in closed-form in \cite[Lemma~A.2]{Kroshnin}.
    In practice, we also incorporate backtracking line search \cite[Section~9.2]{BoydVandenberghe} in order to prevent overshooting, and we find that this helps numerically.

    We do not have rigorous guarantees for the convergence of the PRGD procedure described above.
    Nevertheless, we observe exponentially fast convergence in our experiments; a full convergence theory, as in \cite{BW_GD}, is left as an open problem.

    Now we focus on a particular example so that we can make this discussion more concrete.
    Specifically, let $P$ denote the law of $\Sigma=AA^\top+\frac{1}{2}\Id$ where $A\in\Rbb^{2\times 3}$ is a matrix of IID standard Gaussian entries; note that this is a rotationally-invariant model as described in Example~\ref{ex:invar}.
    Also, let $E_r:=\{M\in\CovSpacePos:\BW(M,M^{\ast})\ge r\}$ for some $r>0$.
    As we showed in the proof of Corollary~\ref{cor:degenerate}, the minimizer of $I_P$ over $E_r=\{M\in\CovSpacePos:\BW(M,M^{\ast})\ge r\}$ must lie on $S_r:=\{M\in\CovSpacePos:\BW(M,M^{\ast})= r\}$ so we can further restrict the optimization in~\eqref{eqn:LDP-num} from $E_r$ to $S_r$.    
    Implementing the zeroth-order method requires discretizing $S_r$, and we do this by independently sampling $K=150$ many directions $A\in\SymSpace_{M^{\ast}}$ with $\|A\|_{2}=1$ uniformly at random and evaluating $I_P$ at $\proj_{\BW}((\Id+A)M^{\ast}(\Id+A);\,S_r)$.
    Implementing the first-order method requires computing $\proj_{\BW}(M;\,S_r)$ for general $M\notin S_r$.
    Thus, for both procedures, we note that, if $\gamma:[0,1]\to\CovSpacePos$ is the constant-speed geodesic connecting $M^{\ast}$ to $M$, then $\BW(\gamma(t),M^{\ast})=t\,\BW(M,M^{\ast})$ hence $\proj_{\BW}(M;\,S_r)=\gamma\big(r/\BW(M,M^{\ast})\big)$.

    The results are shown in Figure~\ref{fig:prgd}.
    We show both the value $I_P(M_i)$ and the excess value $I_P(M_i)-I_P(M_{\infty})$ achieved by the PRGD algorithm with backtracking line search, as the number of iterations ranges over $i=0,1,\ldots, 50$.
    We illustrate the fast convergence of the PRGD algorithm across various choices of $r$.

\begin{figure}[t]
  \centering
  \includegraphics[width=0.95\textwidth]{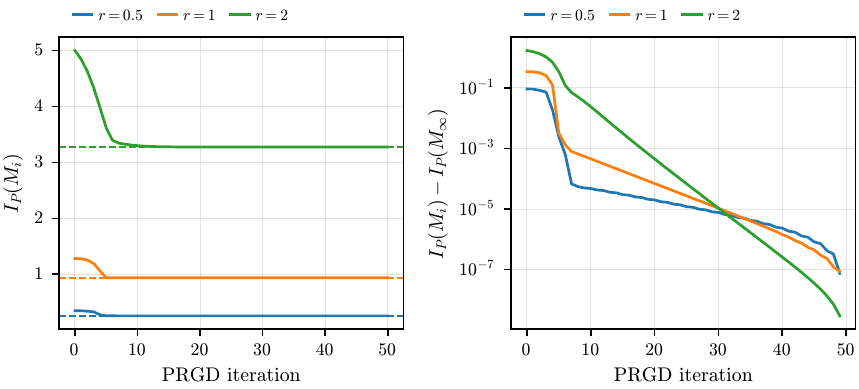}
  \caption{Numerically minimizing the rate function via projected Riemannian gradient descent (PRGD) with backtracking line search.
    In both plots, we show the value $I_P(M_i)$ as a function of the iterations $i=0,1,\ldots, 50$, and we show the result of an independent zeroth-order method which samples $K=150$ many points randomly on the sphere $\{M\in\CovSpacePos: \BW(M,M^{\ast})=r\}$.
    In order to illustrate the convergence of our procedure, we show the values achieved by both methods (left), and we show the logarithm of the difference between the first-order method and its terminal value (right).}

  \label{fig:prgd}
\end{figure}

\section{Extensions}\label{sec:ext}

Our development of large deviations theory for Bures-Wasserstein barycenters suggests a strategy for developing large deviations theory for barycenters in more general geometric settings.
Indeed, observe that our analysis was really only based on a few key ingredients: a fixed-point characterization of barycenters, sufficient integrability of the population distribution, and a notion of exponential tilting on a suitable linear space related to the geometric setting.
In this section we show that this program can be executed---partially or fully---in some other settings of interest.
Specifically, we provide (Subsection~\ref{subsec:wass}) sufficient conditions for establishing a full large deviations principle for barycenters in the univariate Wasserstein space and a partial large deviations principle, in the form of the large deviations upper bound, for barycenters in the general multivariate Wasserstein space, and we provide (Subsection~\ref{subsec:man}) sufficient conditions for establishing the full large deviations theory for barycenters on Riemannian manifolds.

We begin by casting the results of Section~\ref{sec:main} in purely geometric terms.
Indeed, recall that, for each $M\in\CovSpacePos$, the embedding \eqref{eqn:log} can be interpreted as a logarithm map denoted $\ell_M$ which maps the space $\CovSpacePos$ into the Hilbert space $(\SymSpace_M, \langle\,\cdot\,,\,\cdot\,\rangle_M )$.
As such, it is natural to regard $(\SymSpace_M, \langle\,\cdot\,,\,\cdot\,\rangle_M )$ as the \textit{tangent space of $\CovSpacePos$ at $M$}.
From this perspective, we notice that our rate function $I_P$ can be written as
\begin{equation*}
	I_P(M) = \sup_{A\in\SymSpace_{M}}\left(- \log\int_{\CovSpace}\exp\langle A, \ell_{M}(\Sigma)\rangle_{M}\diff P(\Sigma)\right)
\end{equation*}
for $M\in\CovSpacePos$.

\subsection{Wasserstein Space}\label{subsec:wass}

Fix $k\in\Nbb$, and let us write $\mathcal{W}(\Rbb^k):=\Pcal_2(\Rbb^k)$ for the Wasserstein space over $\Rbb^k$, which we endow with the Wasserstein metric $W_2$.
We also write $\wass(\Rbb^k)\subseteq \mathcal{W}(\Rbb^k)$ for the space of probability measures which are absolutely continuous with respect to the Lebesgue measure on $\Rbb^k$.
We define the \textit{Wasserstein barycenter} of $P\in\Pcal_1(\mathcal{W}(\Rbb^k))$, using the usual renormalization trick \cite{schotz2022strong}, as 
\begin{equation*}
    F(P):=\underset{\mu\in\mathcal{W}(\Rbb^k)}{\textnormal{arg min}}\int_{\mathcal{W}(\Rbb^k)}(W_2^2(\mu,\nu)-W_2^2(\nu,\delta_0))\diff P(\nu).
\end{equation*}
When $\nu_1,\nu_2,\ldots$ are IID samples from $P$ then we write $\mu^{\ast}:=F(P)$ for the population Wasserstein barycenter and $\mu_n^{\ast}:=F(\bar P_n)$ for the  empirical Wasserstein barycenter, as usual.
It is natural to hope that the result of the previous section can be shown to hold for Wasserstein barycenters, due to its formal Riemannian structure in the framework of the so-called Otto calculus \cite{otto2001geometry}.

In order to state such a result, we enumerate the following assumptions:
\begin{enumerate}
    \item[\namedlabel{intW2}{\textbf{($\mathcal{W}$E)}}] 
	   For all $\mu\in \Pcal_2(\Rbb^k)$ and $\lambda>0$ we have $\int_{\Pcal_2(\Rbb^k)}\exp(\lambda W_2(\mu,\nu))\diff P(\nu) < \infty$.
    \item[\namedlabel{regW2}{\textbf{($\mathcal{W}$R)}}] We have $P(\wass(\Rbb^k)) = 1$.
\end{enumerate}
Because we have only partial results in general, we consider the univariate and multivariate cases separately in the sequel.

\subsubsection*{Univariate}

First we consider the univariate Wasserstein space (that is, $k=1$) where much of our work is simplified and we can fully establish the large deviations principle.
In fact, the treatment of $\wass(\Rbb)$ is simple because of its ``flat'' geometry: the map
\begin{equation}\label{eqn:W2-quantile-embedding}
\mu\mapsto q_{\mu}
\end{equation}
from $\wass(\Rbb)$ into $L^2([0,1])$, where $q_{\mu}:(0,1)\to\Rbb$ denotes the quantile function of $\mu$, is an isometric embedding and its image is convex.
In particular, the barycenter $\mu^{\ast}$ of $P\in\Pcal_2(\wass(\Rbb))$ is characterized in terms of the (linear) expectation in the embedding space, that is
\begin{equation*}
q_{\mu^{\ast}} = \int_{\wass(\Rbb)}q_{\nu}\diff P(\nu).
\end{equation*}
This explicit representation of the barycenter fills the role of the fixed-point equation in the main body, and leads to the following:

\begin{theorem}
    Under conditions~\textnormal{\hyperref[intW2]{\textbf{($\mathcal{W}$E)}}} and~\textnormal{\hyperref[regW2]{\textbf{($\mathcal{W}$R)}}}, the sequence $\{\mu_n^{\ast}\}_{n\in\Nbb}$ satisfies a large deviations principle in $\wass(\Rbb)$ with good rate function
    \begin{equation*}
    	I_P(\mu) = \sup_{q\in L^2([0,1])}\left(\langle q,q_{\mu} \rangle_{L^2([0,1])} - \log\int_{\wass(\Rbb)}\exp \langle q, q_{\nu}\rangle_{L^2([0,1])} \diff P(\nu)\right),
    \end{equation*}
for $\mu\in\wass(\Rbb)$.
\end{theorem}

\begin{proof}
    By Cram\'er's theorem, the sample means of the quantile functions satisfy an LDP in $L^2([0,1])$ with good rate function given by

    \begin{equation*}
    	I_P(\tilde{q}) = \sup_{q\in L^2([0,1])}\left(\langle q,\tilde{q} \rangle_{L^2([0,1])} - \log\int_{\wass(\Rbb)}\exp \langle q, q_{\nu}\rangle_{L^2([0,1])} \diff P(\nu)\right),
    \end{equation*}
    for any quantile function $\tilde{q}$.
    Since the embedding in~\eqref{eqn:W2-quantile-embedding} is an injective isometry onto a convex set, we use \cite[Lemma~4.1.5(a)]{dembo2009large} and then apply the contraction principle to the inverse map.
\end{proof}

We emphasize that the geometry of $\wass(\Rbb)$ is flat, so this result is proven simply by identifying the Wasserstein barycenters with suitable sample means.
This phenomenon fails for $k\ge 2$, where the problem is genuinely geometric.

\subsubsection*{Multivariate}

Next we consider the case of general $k\in\Nbb$, in which we can only establish a partial result in the form of a large deviations upper bound. We continue with the same notation used for the one-dimensional case, introduced at the beginning of this Subsection, but consider the general multivariate Wasserstein space.

\begin{theorem}
    Under conditions~\textnormal{\hyperref[intW2]{\textbf{($\mathcal{W}$E)}}} and~\textnormal{\hyperref[regW2]{\textbf{($\mathcal{W}$R)}}}, the sequence $\{\mu_n^{\ast}\}_{n\in\Nbb}$ satisfies the large deviations upper bound in $\wass(\Rbb^k)$ with rate function
    \begin{equation}\label{eq:wass:ratefun}
    	I_P(\mu) = \sup_{\alpha\in L^2(\mu)}\left(\langle \alpha,\id \rangle_{L^2(\mu)} - \log\int_{\wass(\Rbb^k)}\exp \langle \alpha, t_{\mu}^{\nu}\rangle_{L^2(\mu)} \diff P(\nu)\right),
    \end{equation}
    where $t_{\mu}^{\nu}:\Rbb^k\to\Rbb^k$ is the optimal transport map from $\mu$ to $\nu$.
\end{theorem}

\begin{proof}
    By \cite{SanovWasserstein} and condition~\textnormal{\hyperref[intW2]{\textbf{($\mathcal{W}$E)}}}, the empirical measures $\{\bar{P}_n\}_{n\in\Nbb}$ defined via $\bar{P}_n:=\frac{1}{n}\sum_{i=1}^{n}\delta_{\nu_i}$ satisfy an LDP in $\Pcal_1(\wass(\Rbb^k),W_2)$ with good rate function $H(\,\cdot\,|\,P)$.
    Also, \cite{KroshninFrechetMonge} and condition~\textnormal{\hyperref[regW2]{\textbf{($\mathcal{W}$R)}}} imply that the barycenter function $F:(\Pcal_1(\wass(\Rbb^k),W_2),W_1)\to(\wass(\Rbb^k),W_2)$ is well-defined and continuous.
    (Alternatively, this follows from our proof of Theorem~\ref{thm:cty-general}).
    Thus, by the contraction principle \cite[Theorem~4.2.1]{dembo2009large}, we find that $\{F(\bar P_n)\}_{n\in\Nbb} = \{\mu_n^{\ast}\}_{n\in\Nbb}$ satisfies an LDP in $(\wass(\Rbb^k),W_2)$ with good rate function
    \begin{equation*}
		J_P(\mu) := \inf\left\{H(Q\,|\,P): Q\in \Pcal_1(\wass(\Rbb^k),W_2), F(Q) = \mu\right\}.
    \end{equation*}
    Now use the fixed-point equation \cite{alvarez2016fixed} for Wasserstein barycenters, which states that $F(Q) = \mu$ implies $\int_{\wass(\Rbb^k)}t_{\mu}^{\nu}\diff Q(\nu) = \id$, along with the Donsker-Varadhan variational formula \cite[Lemma~6.2.13]{dembo2009large} to obtain
    \begin{align*}
		J_{P}(\mu)&=\inf\left\{H(Q\,|\,P): Q\in \Pcal_1(\wass(\Rbb^k),W_2), F(Q) = \mu\right\}\\
        &\ge\inf\left\{H(Q\,|\,P):Q\in \Pcal_1(\wass(\Rbb^k),W_2),  \int_{\wass(\Rbb^k)}t_{\mu}^{\nu}\diff Q(\nu) = \id\right\} \\
        &= \sup_{\alpha\in \Tan_{\mu}(\wass(\Rbb^k))}\left(\langle \alpha,\id \rangle_{L^2(\mu)} - \log\int_{\wass(\Rbb^k)}\exp \langle \alpha, t_{\mu}^{\nu}\rangle_{L^2(\mu)} \diff P(\nu)\right) \\
        &= \sup_{\alpha\in L^2(\mu)}\left(\langle \alpha,\id \rangle_{L^2(\mu)} - \log\int_{\wass(\Rbb^k)}\exp \langle \alpha, t_{\mu}^{\nu}\rangle_{L^2(\mu)} \diff P(\nu)\right) = I_{P}(\mu).
    \end{align*}
    In the last line, we used the tangent space $\Tan_{\mu}(\wass(\Rbb^k))$ which equals the $L^2(\mu)$-closure of the span of $t_{\mu}^{\nu}-\id$, although, since $t_{\mu}^{\nu}$ only appears through $\langle \alpha, t_{\mu}^{\nu}\rangle_{L^2(\mu)}$, we can equivalently take the supremum over $\Tan_{\mu}(\wass(\Rbb^k))$ or $L^2(\mu)$.
\end{proof}

As the proof indicates, the difficulty in extending this result to a bona fide LDP is that the fixed-point equation
\begin{equation*}
\int_{\wass(\Rbb^k)}(t_{\mu}^{\nu}-\id)\diff P(\nu) = 0
\end{equation*}
is a necessary but not sufficient condition in general, without additional a priori constraints.
(See \cite[Example~3.1]{alvarez2016fixed} and the preceding discussion.)
One way to prove the bona fide LDP would be to prove that the barycenter manifests a certain regularity, namely that its density is everywhere positive, but this is difficult in general.

One way to gain such regularity is via entropic regularization \cite[Proposition~3.5]{carlier2021entropic}.
In fact, we expect that one can show that entropically-regularized barycenters satisfy a large deviations principle with some rate function of a form similar to $I_P$. The case of entropically-regularized barycenters is particularly interesting from a geometric perspective, due to the distinctly different Riemannian structure of the regularized Wasserstein space, which has only recently been investigated \cite{lavenant2024riemannian}. 
(Let us also mention a different but related question, which is the large deviations principle for the entropically-regularized transport plan: the work \cite{bernton2022entropic} provides such a result as the regularization parameter goes to zero, but there is no focus on barycenters nor on sample size.)

\subsection{Riemannian Manifolds}\label{subsec:man}

In this last subsection, we point out that the same methods as above can be used to prove a large deviations principle for barycenters on Riemannian manifolds, under suitable hypotheses.
Let $\Mcal$ be a Riemannian manifold with metric $d$ induced by its metric tensor, and let the tangent space of a point $m\in\Mcal$ be denoted as $(\mathcal{T}_m,\langle\,\cdot\,,\cdot\,\rangle_m)$.
For $P\in\Pcal_1(\Mcal)$ we define the \textit{barycenter of $P$} to be
\begin{equation*}
    F(P):=\underset{m\in\Mcal}{\textnormal{arg min}}\int_{\Mcal}(d^2(m,s)-d^2(s_0,s))\diff P(s),
\end{equation*}
where $s_0\in\Mcal$ is fixed (and the barycenter does not depend on the choice of $s_0$).
When $s_1,s_2,\ldots$ are IID samples from $P$ then, as always, we write $m^{\ast}:=F(P)$ for the population barycenter and $m_n^{\ast}:=F(\bar P_n)$ for the empirical barycenter.

Following the geometric principles above, it is natural to expect that barycenters on Riemannian manifolds satisfy a large deviations principle with a suitable good rate function.
To this end, consider the following classical assumptions:
\begin{enumerate}
    \item[\namedlabel{ass_intRM}{\textbf{($\mathcal{M}$E)}}] 
	For all $m\in\Mcal$ and all $\lambda>0$ we have $\int_{\Mcal}\exp(\lambda d(m,s))\diff P(s) < \infty$.
    \item[\namedlabel{regRM}{\textbf{($\mathcal{M}$R)}}] 
	There exists a set $R\subseteq\Mcal$ such that $P(R) = 1$ and such that, if $Q\in\Pcal_1(\Mcal)$ has $Q(R) = 1$, then its barycenter is characterized as the unique solution in $R$ to the fixed-point equation
	\begin{equation*}
		\int_{\Mcal}\log_{m}(s)\diff Q(s) = 0,
	\end{equation*}
	for $m\in \Mcal$, where $\log_m$ is the (Riemannian) logarithm map at $m$.
\end{enumerate}
Notice the parallelism with the assumptions in Section~\ref{sec:main}:
Condition~\ref{ass_intRM} provides the necessary 
exponential integrability, and condition~\ref{regRM} provides the necessary regularity.
While condition~\ref{regRM} may appear cumbersome, it is known to hold in many settings:
if $\Mcal$ has non-positive curvature, then one can take $R=\Mcal$; if $\Mcal:=\mathbb{S}^{m-1}$ is a hypersphere, then one can take $R$ to be any open hemisphere; if $\Mcal:=\CovSpacePos$ (as in the finite-dimensional case of the main body of this paper), then one can just take $R=\CovSpacePos$.
In the most general setting of an abstract Riemannian manifold $\Mcal$, it is known that if the curvature of $\Mcal$ is bounded above by $\kappa>0$, then it suffices to take $R$ to be a ball of sufficiently small radius depending on $\kappa$ \cite[Theorem~B and Corollary~15]{YOKOTA}.
We emphasize that this makes the preceding discussion applicable to other Riemannian geometries on $\CovSpace$ and $\CovSpacePos$, like the affine-invariant and log-Euclidean structures \cite{thanwerdas2023n,arsigny2007geometric, arsigny2006log}.

\begin{theorem}
    Under conditions~\textnormal{\hyperref[ass_intRM]{\textbf{($\mathcal{M}$E)}}} and~\textnormal{\hyperref[regRM]{\textbf{($\mathcal{M}$R)}}}, the sequence $\{m_n^{\ast}\}_{n\in\Nbb}$ satisfies the large deviations principle in $\Mcal$ with good rate function
    \begin{equation*}
    	i_P(m) = \sup_{a\in\mathcal{T}_m}\left(- \log\int_{\Mcal}\exp\langle a, \log_{m}(s)\rangle_{m}\diff P(s)\right)
    \end{equation*}
    for $m\in\Mcal$.
\end{theorem}

Indeed, in order to prove this, one can simply reproduce the steps of Theorem~\ref{thm:main} mutatis mutandis since the log map is injective on $R$, the fixed-point equation characterizes barycenters on $R$, and suitable continuity of the barycenter map is known \cite{schotz2022strong, Relaxation}.

\section*{Acknowledgements}

We thank Victor M.\ Panaretos and Steve N.\ Evans for facilitating this collaboration, and we thank Fraydoun Rezakhanlou for many useful conversations about large deviations theory while the authors were at UC Berkeley.
We also thank Jonathan Niles-Weed and Victor-Emmanuel Brunel for many interesting comments.

%\nocite{*}
\bibliographystyle{amsplain}
\bibliography{references}

\end{document}